\documentclass[11pt]{amsart}
\usepackage[latin1]{inputenc}
\usepackage{amsmath}
\usepackage{amssymb}
\usepackage{todo}
\usepackage{hyperref}
\usepackage{color}
\usepackage{pst-all}

\usepackage{latexsym}
\usepackage{amsmath}
\usepackage{mathtools}

\usepackage[english]{babel}

\usepackage{texdraw}
\usepackage{enumerate}
\usepackage{tikz}
\usepackage{subfigure}
\usetikzlibrary{positioning}
\usepackage{graphicx}

\theoremstyle{plain}

\newtheorem{theorem}{Theorem}[section]
\newtheorem{proposition}[theorem]{Proposition}
\newtheorem{lemma}[theorem]{Lemma}
\newtheorem{corollary}[theorem]{Corollary}
\newtheorem{example}[theorem]{Example}
 \newtheorem{remark}[theorem]{Remark}

\newcommand{\La}{\mathcal{L}} 
\newcommand{\A}{\mathcal{A}}
\renewcommand{\S}{\mathcal{S}}
\newcommand{\B}{\mathcal{B}}
\newcommand{\N}{\mathbb{N}}
\newcommand{\Z}{\mathbb{Z}}

\newcommand{\M}{\mathcal{M}}

\DeclareMathOperator{\E}{\mathcal{E}}
\DeclareMathOperator{\T}{\mathcal{T}}

\DeclareMathOperator{\Card}{Card}
\DeclareMathOperator{\supp}

\begin{document}

\title[On the dimension group  of unimodular $\S$-adic subshifts]{On the dimension group\\  of unimodular $\S$-adic subshifts}

\author{V. Berth\'e}

\address{Universit\'e de Paris, IRIF, CNRS, F-75013 Paris, France}
\thanks{This work was supported by the Agence Nationale de la Recherche  through the project  ``Codys''  (ANR-18-CE40-0007).}
\email{berthe@irif.fr}

\author{P. Cecchi Bernales}
\address{Centro de Modelamiento Matemático, Universidad de Chile, Chile}
\thanks{The  second author was supported by the PhD grant CONICYT - PFCHA / Doctorado Nacional / 2015-21150544.}
\email{pcecchi@ing.uchile.cl}

\author{F. Durand}
\address{LAMFA, UMR 7352 CNRS, Universit\'e de Picardie Jules Verne,
33, rue Saint-Leu, 80039 Amiens, France}
\email{fabien.durand@u-picardie.fr}

\author{J. Leroy}
\address{D\'epartement de math\'ematique, Universit\'e de Li\`ege,
 12, All\'ee de la d\'ecouverte (B37), 4000 Li\`ege, Belgique}
\email{j.leroy@uliege.be}

\author{D. Perrin}
\address{Laboratoire d'Informatique Gaspard-Monge, Universit\'e Paris-Est, France}
\email{Dominique.Perrin@u-pem.fr}

\author{S. Petite}
\address{LAMFA, UMR 7352 CNRS, Universit\'e de Picardie Jules Verne,
33, rue Saint-Leu, 80039 Amiens, France}
\email{samuel.petite@u-picardie.fr}

\date{\today}

\begin{abstract}
Dimension groups are complete invariants of strong orbit equivalence 
for minimal Cantor systems. 
This paper studies  a  natural  family of minimal Cantor systems  having a finitely 
generated dimension group, namely   the primitive  unimodular proper  $\mathcal{S}$-adic subshifts.  They are  generated  by iterating sequences
of substitutions.  Proper substitutions  are such that the    images  of  letters  start   with a same letter, and similarly end with a same letter.
This family  includes various classes of  subshifts  such as  Brun subshifts  or dendric subshifts, that in turn  include   Arnoux-Rauzy subshifts  and natural coding of interval exchange transformations.
We compute their dimension group and  investigate the relation  between 
 the  triviality   of the infinitesimal subgroup and  rational independence  of letter measures.  We also introduce the notion of balanced functions and provide a
  topological characterization of balancedness for primitive unimodular proper S-adic subshifts.
\end{abstract}

\maketitle



\section{Introduction}
Two  dynamical systems are topologically orbit equivalent if there is a homeomorphism between 
them preserving the orbits. Originally, the notion of orbit equivalence was studied  
in the measurable context  (see for instance  \cite{Dye,OrnsteinWeiss}), motivated by  the classification of von Neumann algebras. In contrast
with the measurable case, Giordano, Putnam and Skau showed that, in the topological setting, uncountably 
many classes appear by providing a dimension group    as a complete invariant of  strong 
orbit equivalence   \cite{GPS:95}. 
Dimension groups are ordered direct limit groups defined by sequences of positive homomorphisms $(\theta _n : \Z^{d_n} \to \Z^{d_{n+1}})_n$, 
where $ {\mathbb Z}^d$ is given the standard or simplicial order, and were defined by Elliott \cite{Elliott:76} to study approximately finite dimensional $C^*$-algebras.
In fact, an ordered group is a dimension group if and only if it is a Riesz group \cite{EffrosHandelmanShen:1980}.
They have been widely studied in the late 70's and at the beginning of the 80's \cite{Effros},
in particular when the dimension group is a direct limit given by unimodular matrices \cite{Effros&Shen:1979,Effros&Shen:1980,Effros&Shen:1981,Riedel:1981,Riedel:1981b}.

The present  paper  studies   dynamical and ergodic properties of  subshifts having  dimension groups 
with  a  group  part   of the form ${\mathbb Z}^d$.
We focus on  the class of  primitive  unimodular  proper $\mathcal{S}$-adic subshifts. 
Such subshifts are generated  by iterating   sequences of  substitutions.  They have recently attracted  much attention in symbolic dynamics
\cite{Berthe&Delecroix:14} and in tiling theory \cite{GM:13,Fusion:14}. Proper substitutions are such that images of letters  start with a  same letter and   also
end with a  same letter.
Proper  minimal proper $\S$-adic systems  have played an important role for the characterization of linearly recurrent subshifts \cite{Durand:2000, Durand:03}.
The  term unimodular  refers to  the unimodularity of the  incidence matrices of the associated substitutions.

Sturmian subshifts,  subshifts generated by  natural codings of interval exchange transformations or  Arnoux-Rauzy subshifts are  prominent examples
of  unimodular  proper $\mathcal{S}$-adic subshifts.
They also  belong to a recently defined family of subshifts,  called dendric  subshifts, and   considered in \cite{BDFDLPRR:15,BDFDLPRR:15bis,BFFLPR:2015,BDFDLPR:16,Rigidity} (see also Section~\ref{subsec:tree}).
 In this series of papers, their elements  have been studied  under the name of tree words. We have chosen to use the terminology dendric subshift in order to avoid any ambiguity with respect to shifts defined on trees  (see, e.g.,~\cite{AubrunBeal}) and also to avoid the puzzling term ``tree word''.

Minimal dendric subshifts are  defined with respect to combinatorial properties of their language expressed in terms of extension graphs. For precise definitions, see Section~\ref{subsec:tree}.  In particular, 
they  have linear factor complexity. Focusing on extension properties of factors  is a combinatorial viewpoint that allows to highlight the common features  shared by   dendric subshifts, even if the corresponding symbolic  systems have very distinct  dynamical, ergodic and spectral properties.
For instance,  a  coding of a generic  interval exchange is  topologically weakly mixing  for irreducible permutations not of rotation class \cite{NogRud},   whereas an Arnoux-Rauzy subshift is  generically  not topologically   weakly mixing   \cite{Cassaigne-Ferenczi-Messaoudi:08,BST:2019}.
Even though one can disprove,  in some cases, whether two  given minimal dendric  subshifts are topologically conjugate by using e.g. asymptotic pairs  (see for instance Section \ref{ex:balance}), the question of orbit equivalence is more subtle and is one of the  motivations for the present work.

The aim of this paper is   to study topological  orbit equivalence  and strong orbit equivalence  for  minimal unimodular  proper   $\S$-adic  
subshifts. Let $(X,S)$ be  such a subshift  over a  $d$-letter alphabet ${\mathcal A}$  and 
let ${\mathcal M} (X,S)$ stand for its set of shift-invariant  probability measures.
One of our   main results states  that   any continuous integer-valued function defined on $X$  is cohomologous to some integer linear combination of characteristic functions of letter cylinders 
 (Theorem \ref{theo:cohoword}). 
This relies on the fact that such subshifts being  aperiodic (see Proposition \ref{prop:minimalSSIprimitif}) and recognizable by~\cite{BSTY}, they     have      a sequence  of  Kakutani-Rohlin  tower partitions
with suitable topological  properties.  
We  then  deduce  an explicit computation of their dimension group (Theorem \ref{theo:dg}).
Indeed,  the dimension group $K^{0}(X,S)$ with ordered unit is isomorphic to
$\left( {\mathbb Z} ^d, \,  \{ {\bf x} \in {\mathbb Z}^d \mid \langle {\bf x},  \boldsymbol{ \mu} \rangle > 0 \mbox{ for all }  
 \mu \in  {\mathcal M} (X,S) \}\cup \{{\mathbf 0}\},\,  \bf{1}\right )$,
 where $\boldsymbol{ \mu}$ denotes the vector   of measures of letter cylinders.

In other words, 
strong  orbit equivalence  can be characterized by means of letter measures, i.e.,  by measures of letter  cylinders.
In particular, 
two shift-invariant  probability measures on $(X,S)$  coinciding  on  the  letter cylinders are proved to be  equal (see Corollary \ref{coro:measures}).  
This result extends  a statement initially proved for interval exchanges in~\cite{FerZam:08}; see also \cite{Putnam:89,Putnam:92,GjerdeJo:02} and \cite{BHL:2019,BHL:2020}. 
Moreover, two  primitive unimodular proper ${\mathcal S}$-adic  subshifts are proved to be  strong orbit equivalent if and only if their  simplexes of letter measures coincide up to a unimodular matrix (see Corollary \ref{cor:oe}), with the   simplex of letter measures  being the   $d$-simplex   generated by   the vectors $(\nu[a])_ {a \in \A}$, for  $\nu$  in ${\mathcal M}(X,S)$.

 We also investigate in Section \ref{sec:saturation}  the  triviality of   the   infinitesimal subgroup and  relate it  to the notion of   balance.
 We provide
 a characterization of the   triviality  of the infinitesimal subgroup  for  minimal unimodular  proper   $\S$-adic  
subshifts in  terms of  rational independence of 
measures of letters (see Proposition  \ref{theo:saturated}).  
Inspired by the classical notion of  balance in word combinatorics (see e.g.  references in \cite{BerCecchi:2018}), 
 we  also introduce  the notion of  balanced  functions and provide  a topological  characterization of  this  balance property for   primitive  unimodular proper $\S$-adic subshifts (see Corollary \ref{cor:balanced}).

\bigskip

We briefly describe the contents of this paper.
Definitions and basic notions are recalled in Section~\ref{sec:def},  including,
in particular,  the notions of   dimension group and   orbit equivalence  in Section~\ref{subsec:dim},  and of  image subgroup   in Section~\ref{subsec:cyl}. Primitive  unimodular $\S$-adic subshifts are introduced in Section~\ref{sec:Sadic}, with     dendric subshifts  being  discussed in more details  in Section~\ref{subsec:tree}.
Their dimension groups  are studied  in Section~\ref{sec:proofs2}.   Section~\ref{sec:saturation} is devoted to the study of infinitesimals
 and  their
 connections with  the notion  of  balance. Some examples are handled  in  Section  \ref{sec:examples}.

\bigskip

\paragraph{\bf Acknowledgements}
We would like to thank M. I. Cortez and F. Dolce for stimulating discussions. We also thank warmly  the  referees of this paper for their  careful reading and their 
very useful  comments,  concerning  in particular    the formulation of Corollary \ref{coro:freq}.

\section{First definitions and background }\label{sec:def}


\subsection{Topological dynamical systems} 
By a {\it topological dynamical system}, we mean a pair $(X,T)$
where $X$ is a compact metric space and $T: X \to X$ is a homeomorphism.
It is a {\it Cantor system} when $ X $ is a Cantor space,
that is, $ X $ has a countable basis of its topology which consists of closed and open sets (clopen sets) and does not have isolated points.
This system  is {\em aperiodic} if it does not have periodic points, i.e., points $x$ such that $T^n(x) = x$ for some $n >0$.
It is {\em minimal} if it does not contain any non-empty proper closed  $T$-invariant subset. 
Any minimal Cantor system is aperiodic.
Two topological dynamical systems $(X_1,T_1)$, $(X_2,T_2)$ are {\em conjugate} when there is a {\em conjugacy} between them, i.e., a homeomorphism $\varphi:X_1 \to X_2$ such that $\varphi \circ T_1 = T_2 \circ \varphi$.

A complex number $\lambda$ is  a {\em continuous eigenvalue}  of 
$(X,T)$ if there exists  a  non-zero continuous function $f \colon X  \rightarrow {\mathbb C}$ such that
$f\circ T= \lambda f$. 
An  {\em additive eigenvalue} is a real number $\alpha$ such that
$\exp(2i \pi \alpha)$ is a   continuous eigenvalue.
Let $E(X,T)$ be the  (additive)  group of additive eigenvalues of $(X,T)$.
We consider its rank over ${\mathbb Q}$, i.e.,  the maximal number of rationally independent  elements of 
$E(X,T)$.
Note that $1$ is always an additive eigenvalue and thus ${\mathbb Z} $ is included in $  E(X,T)$.

A probability measure $\mu$ on $X$ is said to be {\em $T$-invariant} if $\mu(T^{-1} A) = \mu(A)$ for every measurable subset $A$ of $ X$.
Let ${\mathcal M} (X,T)$ be the set of all $T$-invariant probability measures on $(X,T)$. 
It is a convex set and any extremal point is called an {\em ergodic} $T$-invariant measure.  It is well known that any topological dynamical  system admits an ergodic invariant measure. 
The set of ergodic $T$-invariant  probability measures is denoted ${\mathcal M}_e (X,T)$.  
Observe that if $(X,T)$ is a minimal Cantor system, then for all clopen $E$ and all $T$-invariant probability measures $\mu$, one has $\mu(E)>0$. 
The topological dynamical system $(X,T)$ is {\em uniquely ergodic} if there exists a unique $T$-invariant probability measure on $X$.
It is said to be   {\em strictly  ergodic} if it is  minimal  and  uniquely ergodic.

The notation  $\chi_{E}$  stands for  the characteristic function of  $E$;
$\N$ stands for the set of non-negative integers ($0  \in \N$).

\subsection{Subshifts} \label{subsec:SD}

Let $\A$ be a finite alphabet of cardinality $d \geq 2$.
Let us denote by $\varepsilon$ the empty word of the free monoid $\A^*$ (endowed with concatenation), and by $\A^{\Z}$ the set of bi-infinite words over $\A$.
For a bi-infinite word $x \in \A^\Z$, and for $i,j \in \Z$ with $i \leq j$, the notation $x_{[i,j)}$ (resp., $x_{[i,j]}$) stands for $x_i \cdots x_{j-1}$ (resp., $x_i \cdots x_{j}$) with the convention $x_{[i,i)} = \varepsilon$.
For a word  $w= w_{1} \cdots w_{\ell} \in \A^\ell$,
its  {\em length} is denoted $|w|$ and equals $\ell$.
We say that a word $u$ is a {\em factor} of a word $w$ if there exist words $p,s$ such that $w = pus$.
If $p = \varepsilon$ (resp., $s = \varepsilon$) we say that $u$ is a {\em prefix} (resp., {\em suffix}) of $w$.  
For a word  $u \in  \A^{*}$,  an index $ 1 \le j \le \ell$ such that $w_{j}\cdots w_{j+|u|-1} =u$ is called an {\em occurence} of $u$ in $w$ and we use the same term for bi-infinite word in $\A^{\Z}$. 
  The number of occurrences of  a  word $u \in \A^*$ in a  finite word $w$ is   denoted as   $|w|_u$.

The set  $\A^{\Z}$ endowed with the product topology of the discrete topology on each copy of $\A$ is topologically a Cantor set. 
The {\em shift map}  $S$ defined by $S \left( (x_n)_{n \in \mathbb{Z}} \right) = (x_{n+1})_{n \in \mathbb{Z}}$ is a homeomorphism of  $\A^{\Z}$. 
A {\em subshift} is a pair $(X,S)$ where $X$ is a closed shift-invariant subset of some $\A^{\Z}$.
It is thus a {\em topological dynamical system}.
Observe that a minimal subshift is aperiodic whenever it is infinite.

The set of factors of a sequence $x \in \A^{\Z}$ is denoted $\La(x)$. 
For a subshift  $(X,S)$ its {\em  language} $\La(X)$ is $\cup_{x\in X} \La(x)$. The {\em factor complexity} $p_X$ of the subshift $(X,S)$ is the function that with $n \in \N$ associates the number $p_X(n)$ of factors of length $n$ in $\La(X)$.
 

Let $w^-, w^+$ be two words. The cylinder $[w^-.w^+]$ is defined  as the set
$\{ x \in X \mid x_{[-|w^-| , |w^+|)} = w^- w^+ \}$. 
It is a clopen set. 
When $w^-$ is the empty word $\varepsilon$, we set $[\varepsilon .w^+] = [w^+]$.

For $\mu$ a $S$-invariant  probability measure, the measure of a factor $w \in \La(X)$ is defined as  the measure of  the cylinder $[w]$.  
The notation  $\boldsymbol{ \mu}$  stands for  the vector $(\mu([a])_{ a \in {\mathcal A} } \in {\mathbb R}^{\mathcal A}$.
The {\em simplex of letter measures} is defined as the $d$-simplex consisting in all the vectors $\boldsymbol{ \mu}$ with $\mu \in {\mathcal M} (X,S)$, i.e., it   consists of all the convex combinations of the vectors $\boldsymbol{ \mu}$ with $\mu \in {\mathcal M}_e (X,S)$.

\subsection{Dimension groups and orbit equivalence}\label{subsec:dim}

Two  minimal Cantor systems $(X_1,T_1)$ and $(X_2,T_2)$ are  {\em orbit equivalent } if there exists a  homeomorphism $\Phi \colon X_1 \rightarrow X_2$
mapping orbits  onto  orbits, i.e.,  for all $x \in X_1$,  one has
$$\Phi ( \{ T_1^n x  \mid n \in {\mathbb Z} \})= \{ T_2^n \Phi (x) \mid n \in {\mathbb Z}\}.$$
This implies that there  exist two maps $n_1 \colon X_1  \rightarrow  {\mathbb Z}$ and $n_2 \colon  X_2 \rightarrow {\mathbb Z}$ (uniquely defined by aperiodicity)  such  that,  for all $x \in X_1$,
$$\Phi \circ T_1 (x)= T_2^{n_1(x)} \circ \Phi(x) \quad  \mbox { and }  \quad  \Phi \circ T_1^{n_2(x)} (x)= T_2 \circ \Phi (x).$$
The  minimal  Cantor systems  $(X_1,T_1)$ and $(X_2,T_2)$  are {\em strongly orbit equivalent}  if   $n_1$ and $n_2$   both  have at most one point  of discontinuity.
For more on the subject, see e.g. \cite{GPS:95}.

There is a powerful and  convenient way to  characterize orbit and strong orbit equivalence in terms of ordered groups and dimension groups due to  \cite{GPS:95}.
An  {\em ordered group} is a pair $(G,G^+)$, where $G$ is  a countable  abelian group   and  $G^+$ is a subset  of $G$, called the {\em positive cone}, satisfying
$$
 G^+ + G^+ \subset G^+, \quad  G^+ \cap (-G^+) = \{ 0 \}, \quad  G^+ - G^+ = G.
$$
We write $a\leq b$ if $b-a \in G^+$, and $ a<b$ if $b-a \in  G^+$ and $b \neq a$. 
An \emph{order ideal}\index{order! ideal} $J$ of an ordered group $(G,G^+)$ 
is a subgroup $J$ of $G$ such that $J=J^+-J^+$ (with $J^+=J\cap G^+$)
and such that $0\le a\le b\in J$ implies $a\in J$.
An ordered group is \emph{simple}\index{simple ordered group}
\index{ordered group!simple} if it has no nonzero
proper order ideals.

An element $u $ in $G^+$  such that,  for all $a $ in $G$, there exists  some non-negative integer $n$ with $a  \leq nu$ is called an {\em order unit}  for $(G,G^+)$. 
Two ordered groups with order unit  $(G_1,G_1^+,u_1)$  and  $(G_2,G_2^+,u_2)$ are {\em isomorphic}   when there exists  a group  isomorphism  
$\phi \colon G_1 \rightarrow G_2$  such that $ \phi(G_1^+)=G_2^+$ and $\phi (u_1)=u_2$.

We say that an ordered group is {\em unperforated} if for all $a\in G$, if $na\in G^+$ for some $n\in\mathbb{N}\setminus\{0\}$, then $a\in G^+$.
Observe that  this implies in particular that $G$ has no torsion element. 
A {\em dimension group} is  an unperforated ordered group with order unit $(G,G^+,u)$ satisfying the {\em Riesz interpolation property}: given $a_1, a_2, b_1, b_2\in G$ with $a_i\leq b_j$ ($i,j=1,2$), there exists $c\in G$ with $a_i\leq c\leq b_j$.

Most examples of dimension groups we will deal with in this paper are of the following type:
$(G, G^+,u ) = ( {\mathbb Z} ^d, \,  \{ {\bf x}\in {\mathbb Z} ^d  \mid \theta_i ( {\bf x}) > 0, \,  1 \leq i \leq e  \} \cup\{\mathbf{0}\},u ) $,
where the $\theta_i$'s are  independent linear forms such that $\theta_i(u)=1$.

Let $(X,T)$ be a  Cantor  system. 
Let $C(X, {\mathbb R})$ and $C(X, {\mathbb Z})$ respectively stand for the group of continuous functions from $X$ to ${\mathbb R}$ and ${\mathbb Z}$, and let $C(X, {\mathbb N})$ stand for the monoid of continuous functions from $X$ to $\mathbb{N}$, with the group and monoid  operation being the addition.
Let
$$
 \begin{array}{lrcl}
  \beta \colon & C(X, {\mathbb Z}) & \rightarrow & C(X, {\mathbb Z}) \\
  & f & \mapsto & f \circ T - f.
 \end{array}
$$
A map $f$ is called  a  {\em  coboundary} (resp., a {\em real coboundary}) if there exists  a map $g$  in $C(X, {\mathbb Z})$ (resp. in $C(X, {\mathbb R})$)   such that $f =  g \circ T -g$.
Two maps $f,g \in C(X,\mathbb{Z})$ are said to be {\em cohomologous} whenever $f-g$ is a coboundary. 

We consider    the quotient group $H(X,T)=C(X, {\mathbb Z})/ \beta C(X,{\mathbb Z})$.
We  denote $[f]$ the class of a function $f$ in $H$, and $\pi$ the natural  projection $\pi \colon C(X, {\mathbb Z}) \rightarrow H(X,T)$.
We define $H^+ (X, T) = \pi (C(X, {\mathbb N}))$ as  the set of classes of functions in $C(X, {\mathbb N})$.
The ordered   group with order unit
$$
 K^0(X,T) := (H(X,T), H^+ (X,T), [1]),
$$
where  $1$ stands for the one constant valued function, 
 is a dimension group  according  to \cite{Putnam:89}, called  the  {\em dynamical dimension group } of $(X,T)$. We will  use   in this paper
    the abbreviated terminology     {\em dimension group  of $(X,T)$}.
The next result shows that any simple dimension group can be realized as the dimension group of  a  minimal Cantor system.
\begin{theorem}
\cite[Theorem 5.4 and Corollary 6.3]{HermanPS:92}
\label{theo:simpleminimal}
Let $(G, G^+ , u)$ be a dimension group with order unit.
It is simple if and only if there exists a minimal Cantor system $(X,T)$ such that $(G, G^+ , u)$ is isomorphic to $K^0(X,T)$.
\end{theorem}

We also define the  set of {\em infinitesimals} of $K^0 (X,T)$ as
$$
 \mbox{Inf} (K^0 (X,T)) = \left \{ [f] \in H(X,T) : \int f d\mu = 0 \mbox { for all } \mu \in {\mathcal M} (X,T)\right\}.
$$
Note that $H(X,T)/\mbox{Inf}(K^0 (X,T))$ with the induced order also determines a dimension group \cite{GPS:95}. We denote it $K^{0} (X,T)/\mbox{\rm Inf} (K^0 (X,T))$.

The  dimension groups $K^{0} (X,T)$ and $K^{0} (X,T)/\mbox{Inf} (K^0 (X,T))$  are complete invariants   of  strong  orbit equivalence and orbit equivalence, respectively.

\begin{theorem}\cite{GPS:95}\label{oe}
Let $(X_1,T_1)$ and $(X_2,T_2)$ be two minimal Cantor systems.
The following are equivalent:
\begin{itemize}
 \item $(X_1,T_1)$ and $(X_2,T_2)$ are strong orbit equivalent;
 \item $K^{0} (X_1,T_1)$ and $K^0 (X_2,T_2)$ are isomorphic.
\end{itemize}
Similarly, the following are equivalent:
\begin{itemize}
 \item $(X_1,T_1)$  and $(X_2,T_2)$ are orbit equivalent;
 \item $K^{0} (X_1,T_1)/\mbox{\rm Inf} (K^0 (X_1,T_1))$ and $K^{0} (X_2,T_2)/\mbox{\rm Inf} (K^0 (X_2,T_2))$ are isomorphic.
\end{itemize}
\end{theorem}

\subsection{Image   subgroup}
\label{subsec:cyl}

 A {\em trace}  (also called state) of  a dimension group $(G,G^+,u)$ is a group homomorphism $p:G\to \mathbb{R}$ such that $p$ is non-negative ($p(G^+)\geq 0$) and $p(u)=1$. 
The collection of all traces of $(G,G^+,u)$ is denoted by $\T(G,G^+,u)$. 
It is known \cite{Effros} that $\T(G,G^+,u)$ completely determines the order on $G$, if  the dimension group is simple.
In fact, one has 
\[
G^+=\{a\in G: p(a)>0, \forall p\in \T(G,G^+,u)\}\cup \{0\}.
\]  
For more on the subject, see e.g. \cite{Effros}.

Let $(X,T)$ be a  minimal Cantor system. 
Given $\mu \in {\mathcal M} (X,T)$, we define the {trace}  $\tau_{\mu}$ on $K^0 (X,T)$  as  $\tau_{\mu} ([f]):= \int f d\mu$. 
It is shown in \cite{HermanPS:92} that the correspondence $\mu\mapsto \tau_\mu$ is an affine isomorphism from ${\mathcal M} (X,T)$ onto $\T(K^0(X,T))$.
Thus it sends the extremal points of ${\mathcal M} (X,T)$, i.e.,  the ergodic measures, to the extremal points of $\T(K^0(X,T))$, 
called {\em pure traces}.

The {\em image subgroup} of $K^0 (X,T)$ is defined as  the 
 ordered group with order unit
$$
 (I(X,T), I(X,T) \cap {\mathbb R} ^+, 1),
$$
where 
$$
 I(X,T) = \bigcap_{ \mu \in {\mathcal M} (X,T)} \left\{ \int f d\mu \,: \, f \in C(X,{\mathbb Z})\right\}.
$$
Actually,  $E(X,T)$ is a subgroup of $I(X,T)$  (see~\cite[Proposition 11]{CortDP:16} and also   \cite[Corollary 3.7]{GHH:18}.

If $(X,T)$ is uniquely ergodic with unique $T$-invariant probability  measure $\mu$, then $K^0 (X,T)/\mbox{Inf} (K^0 (X,T))$ is isomorphic to $(I(X,T), I(X,T) \cap {\mathbb R} ^+, 1)$, via the correspondence $$[f]+\mbox{Inf} (K^0 (X,T))\mapsto \int fd\mu.$$

Let us recall the following description of $I(X,T)$.
\begin{proposition}{ \cite[Corollary 2.6]{GHH:18}, \cite[Lemma 12]{CortDP:16}.}
\label{prop:GW}
Let $(X,T)$ be a minimal Cantor system.
Then
$$
I(X,T) = \left\{ \alpha : \exists g \in C(X,\mathbb{Z} ) , \alpha =  \int g d \mu \ \forall \mu \in \mathcal{M} (X,T) \right\}.
$$
\end{proposition}

We give an other description of $I(X,T)$ using the following well known lemma.
\begin{lemma}{\cite[Lemma 2.4]{GW}}
\label{lemma:GW}
Let $(X,T)$ be a minimal Cantor system. 
Let $f\in C (X, \mathbb{Z} )$ such that $\int_X f  d \mu $ belongs to $]0,1[$ for every $\mu \in \mathcal{M} (X , T)$.
Then, there exists a clopen set $U$ such that $\int_X f  d \mu = \mu (U)$ for every $\mu \in \mathcal{M} (X , T)$.
\end{lemma}

For a family of real
numbers $N = \{\alpha_{i} : i \in J\}$, we let $\langle N \rangle$ denote the abelian additive group generated by these real numbers. 

\begin{proposition}\label{prop:cob}
Let $(X,T)$ be a minimal Cantor system.
Then,
\begin{equation}\label{eq:I}
I(X,T)
 =  \left\langle \{ \alpha : \exists \  \hbox{\rm  clopen set } U \subset  X ,  \  \alpha =  \mu (U) \   \forall  \mu \in \mathcal{M}(X,T)\}\right\rangle.
\end{equation}
\end{proposition}
\begin{proof}
There is just one inclusion to prove. 
Let $\beta$ be  in $I(X , T )$.
If $\beta$ is an integer then using that $\beta =  \beta\mu (X)$ for all $\mu \in \mathcal{M} (X,T)$, it implies that $\beta$ belongs to the right member of the equality in \eqref{eq:I}.
Otherwise, let $n\in \mathbb{Z}$ be such that $\beta-n$ belongs to $]0,1[$.
From Proposition \ref{prop:GW} and Lemma \ref{lemma:GW} there exists a clopen set $U$ such that $\mu (U) = \beta-n$ for all $\mu \in \mathcal{M} (X,T)$.
It follows  $\mu (U)$ and $n$ belong to 
$\left\langle \{ \alpha : \exists \  \hbox{\rm  clopen set } U \subset  X ,  \  \alpha =  \mu (U) \   \forall  \mu \in \mathcal{M}(X,T)\}\right\rangle $. So it is also the case for $\beta$.
 \end{proof}

One can obtain a  more explicit  description of the set $I (X,S)$ for minimal 
subshifts.

\begin{proposition}\label{prop:cob}
Let $(X,S)$ be a minimal subshift.
Then,
$$\begin{array}{ll}
I(X,S)&=\bigcap_{\mu\in \mathcal{M}(X,S)} \left\langle \{\mu([w]): w\in\mathcal{L}(X)\}\right\rangle .
\end{array}$$


In particular, if $(X,S)$ is uniquely ergodic with $\mu$ its unique $S$-invariant probability measure, then 
$I(X,S)=\left\langle \{\mu([w]): w\in\mathcal{L}(X)\}\right\rangle.$

\end{proposition}
\begin{proof}
The proof  of the  first  equality is  a direct consequence of the fact that   every function belonging to $C(X,\mathbb{Z})$ is cohomologous to some cylinder function in $C(X,\mathbb{Z})$, i.e.,
to some function  $h$  in $ C(X,\mathbb{Z})$   for which there exists $n>0$ such that for all $x\in X$, $h(x)$ depends only on $x_{[0,n)}$. 
Indeed, let $f\in C(X,\mathbb{Z})$. Since $f$ is integer-valued, it is locally constant, and  by compactness  of  $X$, there exists $k\in\mathbb{N}$ such that for all $x\in X$, $f(x)$ depends only on $x_{[-k,k]}$. Therefore, $g(x)=f\circ S^k(x)$ belongs to $C(X,\mathbb{Z})$ and depends only on $x_{[0,2k]}$ for all $x\in X$.
 Finally, $f-g=f-f\circ S^k(x)= f- f\circ S+ f\circ S - f \circ S^2 + \cdots +  f\circ S^{k-1}(x)+ f\circ S^{k}(x) $ is a coboundary. 
Hence,  $\int fd\mu=\int gd\mu$. Since $g$ is a cylinder function,  $g$ can be written as a finite sum of the form
$g=\sum\ell_u\chi_{[u]},$  $u\in \mathcal{L}(X)$ and   $\ell_u\in\mathbb{Z}$. Thus, 
$\int f d\mu=\sum_{}\ell_u\mu([u])\in \left\langle \{\mu([w]): w\in\mathcal{L}(X)\}\right\rangle.$
\end{proof}

\section{Primitive unimodular proper $\S$-adic subshifts}\label{sec:Sadic}


In this section we first  recall  the notion of primitive unimodular proper  $\S$-adic subshift in Section~\ref{subsec:sadic}.
We then  illustrate  it  with  the class of  minimal dendric subshifts in Section \ref{subsec:tree}.

\subsection{$\S$-adic subshifts} \label{subsec:sadic}
 Let $\mathcal{A},\, \mathcal{B}$ be finite alphabets and let $\tau:\, \mathcal{A}^*\to \mathcal{B}^*$ be a \emph{non-erasing} morphism (also called a \emph{substitution} if $\mathcal{A} = \mathcal{B}$). Let us  note  that a morphism is uniquely determined by  its  values on the alphabet ${\mathcal A}$
 and this will be the way we will define them (see e.g. Example \ref{ex:fibo}).  By non-erasing, we mean that the image of any letter is a non-empty word.  We stress the fact that all morphisms are assumed to be non-erasing in the following.
Using concatenation, we extend $\sigma$ to~$\mathcal{A}^\mathbb{N}$ and~$\mathcal{A}^\mathbb{Z}$. 
With a morphism $\tau :\A^* \to \B^*$, where $\A$ and $\B$ are finite alphabets, we classically associate an {\em incidence matrix} $M_\tau$ indexed by $\B \times \A$ such that  for every $(b,a)\in \B \times \A$, its entry at position $(b,a)$ is the number of occurrences of $b$ in $\tau(a)$.
Alphabets are always assumed to have cardinality at least 2.
The morphism $\tau$ is said to be {\em left proper} (resp. {\em right proper}) when there exist a letter $b \in \B$ such that for all $a \in \A$, $\tau(a)$ starts with $b$ (resp., ends with $b$).
It is said to be {\em proper} if it is both left and right proper.

We recall the definition of an $\S$-adic subshift as stated in \cite{BSTY}, see also \cite{Berthe&Delecroix:14} for more on $\S$-adic subshifts.
Let $\boldsymbol{\tau} = (\tau_n : \mathcal{A}_{n+1}^* \to \mathcal{A}_n^*)_{n \geq 1}$ be a sequence of morphisms such that 
$
\max_{a \in \A_n} |\tau_1 \circ \cdots \circ \tau_{n-1}(a)|
$
goes to infinity when $n$ increases.   By non-erasing, we mean that the image of any letter is a non-empty word.
For $1\leq n<N$, we define $\tau_{[n,N)} = \tau_n \circ \tau_{n+1} \circ \dots \circ \tau_{N-1}$ and $\tau_{[n,N]} = \tau_n \circ \tau_{n+1} \circ \dots \circ \tau_{N}$. 
For $n\geq 1$, the \emph{language $\mathcal{L}^{(n)}({\boldsymbol{\tau}})$ of level $n$ associated with $\boldsymbol{\tau}$} is defined~by 
\[
\mathcal{L}^{(n)}({\boldsymbol{\tau}}) = \big\{w \in \mathcal{A}_n^* \mid \mbox{$w$ occurs in $\tau_{[n,N)}(a)$ for some $a \in\mathcal{A}_N$ and $N>n$}\big\}.
\]

As $\max_{a \in \A_n} |\tau_{[1,n)}(a)|$ goes to infinity when $n$ increases, $\mathcal{L}^{(n)}({\boldsymbol{\tau}})$ defines a non-empty subshift $X_{\boldsymbol{\tau}}^{(n)}$ that we call the {\em subshift generated by $\mathcal{L}^{(n)}({\boldsymbol{\tau}})$}.
More precisely, $X_{\boldsymbol{\tau}}^{(n)}$ is the set of points $x \in \mathcal{A}_n^\mathbb{Z}$ such that $\mathcal{L} (x) \subseteq \mathcal{L}^{(n)}({\boldsymbol{\tau}})$.  Note that it may happen that $\mathcal{L}(X_{\boldsymbol{\tau}}^{(n)})$ is strictly contained in $\mathcal{L}^{(n)}({\boldsymbol{\tau}})$.
We set $\mathcal{L}({\boldsymbol{\tau}}) = \mathcal{L}^{(1)}({\boldsymbol{\tau}})$,$X_{\boldsymbol{\tau}} = X_{\boldsymbol{\tau}}^{(1)}$ and call $(X_{\boldsymbol{\tau}},S)$ the \emph{$\S$-adic subshift} generated by the \emph{directive sequence}~$\boldsymbol{\tau}$.

We say that  $\boldsymbol{\tau}$ is {\em primitive} if, for any $n\geq 1$, there exists $N>n$ such that 
$M_{\tau_{[n,N)}} >0$, {\em i.e.}, for all $a \in \A_N$, $\tau_{[n,N)}(a)$ contains occurrences of all letters of $\A_{n}$.
Of course, $M_{\tau_{[n,N)}}$ is equal to  $M_{\tau_n} M_{ \tau_{n+1}}\cdots  M_{\tau_{N-1}}$.
Observe that if $\boldsymbol{\tau}$ is primitive, then $\min_{a \in \A_n} |\tau_{[1,n)}(a)|$ goes to infinity when $n$ increases.
In  the primitive case $\mathcal{L}(X_{\boldsymbol{\tau}}^{(n)})= \mathcal{L}^{(n)}({\boldsymbol{\tau}})$,  and $X_{\boldsymbol{\tau}}^{(n)}$ is a minimal subshift (see for instance ~\cite[Lemma 7]{Durand:2000}).

We say that $\boldsymbol{\tau}$ is ({\em left, right}) {\em proper} whenever each morphism $\tau_n$ is (left, right) proper.
We  also say that $\boldsymbol{\tau}$ is {\em unimodular} whenever, for all $n \geq 1$, $\A_{n+1} = \A_n$ and the matrix $M_{\tau_n}$ has determinant of absolute value 1.
By abuse of language, we say that a subshift is a (unimodular, left or right proper, primitive) $\S$-adic subshift if there exists a (unimodular, left or right proper, primitive) sequence of morphisms $\boldsymbol{\tau}$ such that $X = X_{\boldsymbol{\tau}}$.

Let us give another way to define $X_{\boldsymbol{\tau}}$ when $\boldsymbol{\tau}$ is primitive and proper. 
For an endomorphism $\tau$ of $\A^*$, let $\Omega (\tau ) = \overline{\bigcup_{k\in \Z} S^k \tau ( \A^\Z )}$.

\begin{lemma}
\label{lemma:omega}
Let $\boldsymbol{\tau} = (\tau_n : \mathcal{A}_{n+1}^* \to \mathcal{A}_n^*)_{n \geq 1}$ be a sequence of morphisms such that $\min_{a \in \A_n} |\tau_{[1,n)}(a)|$ goes to infinity when $n$ increases.
Then, 
$$
X_{\boldsymbol{\tau}} \subset \bigcap_{n\in \N} \Omega (\tau_{[1,n]} ).
$$
Furthermore, when $\boldsymbol{\tau}$ is primitive and proper, then the equality  $
X_{\boldsymbol{\tau}} = \bigcap_{n\in \N} \Omega (\tau_{[1,n]} )
$ holds.
\end{lemma}
\begin{proof}
The proof is left to the reader.
\end{proof}

With a left proper morphism $\sigma:\A^* \to \B^*$ such that $b \in \B$ is the first letter of all images $\sigma(a)$, $a \in \A$, we associate the right proper morphism $\overline{\sigma}:\A^* \to \B^*$ defined by 
$b\overline{\sigma}(a) = \sigma(a)b$ for all $a \in \A$.
For all $x \in A^\mathbb{Z}$, we thus have $\bar \sigma(x) = S \sigma(x)$.
The next result is a weaker version of~\cite[Corollary 2.3]{Durand&Leroy:2012}.

\begin{lemma}
\label{lemma:proper}
Let $(X,S)$ be an $\S$-adic subshift  generated by the
primitive and left proper directive sequence $\boldsymbol{\tau} = (\tau_n:\A_{n+1}^* \to \A_n^*)_{n \geq 1}$.
Then $(X,S)$  is also generated by 
the primitive and proper directive sequence $\tilde{\boldsymbol{\tau}} = (\tilde{\tau}_n)_{n \geq 1}$, where for all $n$, $\tilde{\tau}_n = \tau_{2n-1} \overline{\tau}_{2n}$.
In particular,  if $\boldsymbol{\tau}$ is unimodular, then so is $\tilde{\boldsymbol{\tau}}$.
\end{lemma}
\begin{proof}
Each morphism $\tilde{\tau}_n$ is trivially proper.
It is also clear that the unimodularity and the primitiveness of $\boldsymbol{\tau}$ are preserved in this process.
Using the relation $\bar \sigma(x) = S \sigma(x)$ and Lemma~\ref{lemma:omega}, we then get
\[
	X_{\boldsymbol{\tau}} 
	\subset \bigcap_{n \in \mathbb{N}} \Omega (\tau_{[1,n]})
 	= \bigcap_{n \in \mathbb{N}} \Omega (\tilde{\tau}_{[1,n]})
	= X_{\tilde{\boldsymbol{\tau}}}.
\] 
Since both $\boldsymbol{\tau}$ and $\tilde{\boldsymbol{\tau}}$ are primitive, the subshifts $X_{\boldsymbol{\tau}}$ and $X_{\tilde{\boldsymbol{\tau}}}$ are minimal, hence they are equal.
\end{proof}

\begin{lemma}
\label{lemma:aperiodicity}
All   primitive unimodular proper $\S$-adic subshifts are aperiodic. 
\end{lemma}
\begin{proof}
Let $\boldsymbol{\tau}$ be a primitive unimodular proper directive sequence on the alphabet $\A$ of cardinality $d \geq 2$. 
Suppose  that it has a periodic point $x$ of period $p$, where $p$ is the smallest period of $x$ ($p>0$).
By primitiveness, all letters of $\A$ occur in $x$, so $p \geq d$. 
We have $x = \cdots uu.uu\cdots $ for some word $u$ with $|u| = p$.
There exists some $n$ such that, for all $a $, one has $\tau_{[1,n]} (a) = s(a) u^{q(a)}p(a)$, where $s(a), p(a)$ are a strict suffix and a strict prefix of $u$ and $q(a) >1$.
Let $b\in \mathcal{A}$ and set $\tau_{n+1} (b) = b_0 b_1 \cdots b_k$. 
As the directive sequence $\boldsymbol{\tau}$ is proper,  $b_0 b_1 \cdots b_k b_0$ is also a word in $\mathcal{L}^{(n+1)}  (\boldsymbol{\tau})$.
By a classical argument due to Fine and Wilf \cite{Fine:65}, one  has $p(b_0)s(b_1) = p(b_i)s(b_{i+1})= p(b_k)s(b_0) = u$  for $1\leq i\leq k-1$.
Hence
$$
|\tau_1 \cdots \tau_{n+1} (b) |
\equiv
|s(c)p(c)s(b_1)p(b_1) \cdots  s(b_k)p(b_k)|
\equiv
0 \mbox{ modulo }  |u|,
$$ 
which contradicts the  unimodularity of  $\boldsymbol{\tau}$.
\end{proof}

The next two results will be important  for the  computation of  the dimension group of primitive unimodular proper $\S$-adic subshifs.
The first one is a weaker version of~\cite[Theorem 3.1]{BSTY}.

\begin{theorem}[\cite{BSTY}]
\label{theo:BSTY}
Let $\tau :\A^* \to \B^*$ be  such that  its incidence  matrix $M_\tau$ is unimodular.
Then, for any aperiodic $y \in \B^\Z$, there exists at most one $(k,x) \in \N \times A^\Z$ such that $y = S^k \tau(x)$, with $0 \leq k < |\tau(x_0)|$.
\end{theorem}

\begin{proposition}
\label{prop:minimalSSIprimitif}
Let $\boldsymbol{\tau} = (\tau_n : \A^* \to \A^*)_{n \geq 1}$ be a unimodular proper sequence of morphisms such that $\max_{a \in \A} |\tau_{[1,n)}(a)|$ goes to infinity when $n$ increases.
Then $(X_{\boldsymbol{\tau}},S)$ is aperiodic and minimal if and only if $\boldsymbol{\tau}$ is primitive.
\end{proposition}

\begin{proof}
Recall that any $\S$-adic subshift with a primitive directive sequence is minimal (see, e.g.~\cite[Lemma 7]{Durand:2000}) and that aperiodicity is proved in Lemma \ref{lemma:aperiodicity}.

We only have to show that the condition is necessary.
We assume that $(X_{\boldsymbol{\tau}},S)$ is aperiodic and minimal.
For all $n \geq 1$, $(X_{\boldsymbol{\tau}}^{(n)},S)$ is trivially aperiodic. 
Let us show that it is minimal.

Assume by contradiction that for some $n \geq 1$, $(X_{\boldsymbol{\tau}}^{(n)},S)$ is minimal, but not $(X_{\boldsymbol{\tau}}^{(n+1)},S)$.
There exist $u \in \mathcal{L}(X_{\boldsymbol{\tau}}^{(n+1)})$ and $x \in X_{\boldsymbol{\tau}}^{(n+1)}$ such that $u$ does not occur in $x$.
By Theorem~\ref{theo:BSTY}, $\{\tau_n([v]) \mid v \in \La(X_{\boldsymbol{\tau}}^{(n+1)}) \cap \A^{|u|}\}$ is a finite clopen  partition of 
$\tau_n(X_{\boldsymbol{\tau}}^{(n+1)})$.
Thus, considering $y = \tau_n(x)$, by minimality of $(X_{\boldsymbol{\tau}}^{(n)},S)$, there exists $k \geq 0$ such that $S^k y$ is in $\tau_n([u])$.
Take $z \in [u]$ such that $S^ky = \tau_n(z)$.
Since $y$ is aperiodic and since we also have $S^k y = S^{k'} \tau_n(S^\ell x)$ for some $\ell \in \N$ and $0 \leq k' < |\tau_n(x_\ell)|$,
we  obtain  that $\tau_n(z) = S^{k'} \tau_n(S^\ell x)$ with $z \in [u]$, $S^\ell x \notin [u]$ and $0 \leq k' < |\tau_n(x_\ell)|$; this contradicts Theorem~\ref{theo:BSTY}.

We now show that $\lim_{n \to +\infty} \min_{a \in \A} |\tau_{[1,n)}(a)| = +\infty$.
We again proceed by contradiction, assuming that $(\min_{a \in \A} |\tau_{[1,n)}(a)|)_{n \geq 1}$ is bounded.
Then there exists $N > 0$ and a sequence $(a_n)_{n \geq N}$ of letters in $\A$ such that for all $n \geq N$, $\tau_n(a_{n+1}) = a_n$.
We claim that there are arbitrary long words of the form $a_N^k$ in $\La(X_{\boldsymbol{\tau}}^{(N)})$ which contradicts the fact that $(X_{\boldsymbol{\tau}}^{(N)},S)$ is minimal and aperiodic.
Since $\boldsymbol{\tau}$ is proper, for all $n \geq N$ and all $b \in \A$, $\tau_n(b)$ starts and ends with $a_n$.
As $\max_{a \in \A} |\tau_{[1,n)}(a)|$ goes to infinity, there exists a sequence $(b_n)_{n \geq N}$ of letters in $\A$ such that $|\tau_{[N,n)}(b_n)|$ goes to infinity and for all $n \geq N$, $b_n$ occurs in $\tau_n(b_{n+1})$.
This implies that there exists $M \geq N$ such that for all $n \geq M$, $b_n \neq a_n$ and, consequently, that $\tau_n(b_{n+1}) = a_n u_n$ for some word $u_n$ containing $b_n$.  
It is then easily seen that, for all $k \geq 1$, $a_M^k$ is a prefix of $\tau_{[M,M+k)}(b_{M+k})$, which proves the claim.

We finally show that $\boldsymbol{\tau}$ is primitive.
If not, there exist $N \geq 1$ and a sequence $(a_n)_{n \geq N}$ of letters in $\A$ such that for all $n > N$, $a_N$ does not occur in $\tau_{[N,n)}(a_n)$.
As $(|\tau_{[N,n)}(a_n)|)_n$ goes to infinity, this shows that there are arbitrarily long words in $\La(X_{\boldsymbol{\tau}}^{(N)})$ in which $a_N$ does not occur.
Since $\boldsymbol{\tau}$ is unimodular, there is also a sequence $(a'_n)_{n \geq N}$ of letters in $\A$ such that $a_N = a_N'$ and for all $n \geq N$, $a_n'$ occurs in $\tau_n(a_{n+1}')$.
Again using the fact that $|\tau_{[N,n)}(a'_n)|$ goes to infinity, this shows that $a_N$ belongs to $\La(X_{\boldsymbol{\tau}}^{(N)})$.
We conclude that $(X_{\boldsymbol{\tau}}^{(N)},S)$ is not minimal, a contradiction.
\end{proof}

\subsection{Dendric subshifts} \label{subsec:tree}
We  now  describe a   subclass of the    family of  primitive unimodular  proper   $\S$-adic  subshifts,  namely
the  class of  dendric subshifts, that encompasses Sturmian subshifts, Arnoux-Rauzy subshifts  (see  Section \ref{subsec:AR}), as well as subshifts generated by interval exchanges (see~\cite{BDFDLPRR:15}). The ternary  words   generated  by  the  Cassaigne-Selmer  multidimensional continued fraction algorithm  also provide dendric  subshifts~\cite{ArnouxLa:2018,CLL:17}. 

Dendric subshifts are  defined with respect to combinatorial properties of their language expressed in terms of extension graphs. 
We recall the   notion of  dendric  words and subshifts,  studied in  \cite{BDFDLPRR:15,BDFDLPRR:15bis,BFFLPR:2015,BDFDLPR:16,Rigidity}. 
Let $(X,S)$ be a minimal subshift defined on the alphabet $\A$.
For $w \in {\mathcal L}_X$, let
$$
 \begin{array}{lcl}
 L(w) = \{ a \in {\mathcal A} \mid aw \in   {\mathcal L}_X\},							& &	\ell(w) = \Card(L(w)),	\\
 R(w) = \{ a \in {\mathcal A} \mid wa \in {\mathcal L}_X\},							& &	r(w) = \Card(R(w)).
 \end{array}
$$
A word $w \in \La_X$ is said to be {\em right special} (resp. {\em left special}) if $r(w)\geq 2$ (resp. $\ell(w)\geq 2$). It is {\em bispecial} if it is both left and right special.

For a word $w \in \La(X)$, we consider the undirected bipartite graph $\E(w)$ called its \emph{extension graph} with respect to $X$ and defined as follows:
its set of vertices is the disjoint union of $L(w)$ and $R(w)$ and its edges are the pairs $(a,b) \in L(w) \times R(w)$ such that $awb \in \mathcal{L}(X)$.
For an illustration, see Example~\ref{ex:fibo} below. 
We then say that a subshift $X$ is a \emph{dendric subshift} if, for every word $w \in {\mathcal L}(X)$, the graph $\E(w)$ is a tree.
Note that the extension graph associated with every non-bispecial word is trivially a tree.
We will consider here only minimal dendric subshifts.

The factor complexity of a dendric subshift over a $d$-letter alphabet  is $(d-1)n + 1$ (see~\cite{BDFDLPR:16}),
 and on a two-letter alphabet, the  minimal dendric subshifts are the Sturmian subshifts.
 Thus minimal dendric subshift are aperiodic when $d$ is greater or equal to $2$.

\begin{example}\label{ex:fibo}
\rm
Let $\sigma$ be the Fibonacci substitution defined over the alphabet  $\{a,b\}$ by  $\sigma \colon a \mapsto ab, b \mapsto a$  and consider the  subshift generated by $\sigma$ (i.e.,  the set of  bi-infinite words   over $\A$   whose   factors 
belong to some $\sigma^n (a)$).
The extension graphs of the empty word and of the two letters $a$ and $b$ are represented in Figure~\ref{fig:fibo-ext}.

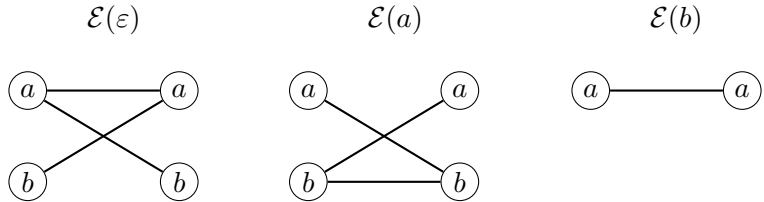
\begin{figure}[h]
 \tikzset{node/.style={circle,draw,minimum size=0.5cm,inner sep=0pt}}
 \tikzset{title/.style={minimum size=0.5cm,inner sep=0pt}}

 \begin{center}
  \begin{tikzpicture}
   \node[title](ee) {$\E(\varepsilon)$};
   \node[node](eal) [below left= 0.5cm and 0.6cm of ee] {$a$};
   \node[node](ebl) [below= 0.7cm of eal] {$b$};
   \node[node](ear) [right= 1.5cm of eal] {$a$};
   \node[node](ebr) [below= 0.7cm of ear] {$b$};
   \path[draw,thick]
    (eal) edge node {} (ear)
    (eal) edge node {} (ebr)
    (ebl) edge node {} (ear);
   \node[title](ea) [right = 3cm of ee] {$\E(a)$};
   \node[node](aal) [below left= 0.5cm and 0.6cm of ea] {$a$};
   \node[node](abl) [below= 0.7cm of aal] {$b$};
   \node[node](aar) [right= 1.5cm of aal] {$a$};
   \node[node](abr) [below= 0.7cm of aar] {$b$};
   \path[draw,thick]
    (aal) edge node {} (abr)
    (abl) edge node {} (aar)
    (abl) edge node {} (abr);
   \node[title](eb) [right = 3cm of ea] {$\E(b)$};
   \node[node](bal) [below left= 0.5cm and 0.6cm of eb] {$a$};
   \node[node](bar) [right= 1.5cm of bal] {$a$};
   \path[draw,thick]
    (bal) edge node {} (bar);
  \end{tikzpicture}
 \end{center}

 \caption{The extension graphs of $\varepsilon$ (on the left), $a$ (on the center) and $b$ (on the right) are trees.}
 \label{fig:fibo-ext}
\end{figure}
\end{example}

The following theorem states a structural property
  for return words of  minimal  dendric subshifts, from which  a description  as primitive unimodular proper $\S$-adic subshifts   can be  deduced (Proposition~\ref{prop:DendricareSadic} below).
Let $(X,S)$ be a minimal subshift over the alphabet $\A$ and let $w \in \La(X)$.
A {\em return word} to $w$ is a word $v$ in $\La(X)$ such that $w$ is a prefix of $vw$ and $vw$ contains exactly two occurrences of $w$.
We recall below a corollary of  \cite[Theorem 4.5]{BDFDLPR:16}.

\begin{theorem}\label{theo:return}
Let $(X,S)$ be a minimal dendric subshift defined on the alphabet $\A$. Then, for any $w \in \La(X)$, the set of return words to $w$ is a basis of the free group  on $\A$.
\end{theorem}
 In particular,  dendric subshifts have   bounded topological rank.
 The next result shows that minimal dendric subshifts are primitive unimodular proper $\S$-adic subshifts.
 Similar results are proved with the same method in~\cite{BFFLPR:2015,Rigidity,BSTY} but not highlighting all  the properties  stated below, so we provide a proof for the sake of self-containedness. 
 It relies on $\mathcal{S}$-adic representations built from return words \cite{Durand:2000,Durand:03} together with the remarkable property of return words of dendric subshifts stated in Theorem \ref{theo:return}.
 We  also provide in 
 Section \ref{subsection:vs}  an example of a primitive unimodular proper subshift which is not dendric and whose strong orbit equivalence class contains no dendric subshift.

\begin{proposition}\label{prop:DendricareSadic}
Minimal dendric subshifts are primitive unimodular proper $\S$-adic subshifts. 
\end{proposition}
\begin{proof}
Let $(X,S)$ be a minimal dendric subshift over the alphabet ${\mathcal A}= \{1,2, \dots,d\}$  and take any $x\in X$. For every $n\geq 1$, let  $V_n(x):=\{v_{1,n},\cdots,v_{d,n}\}$ be the set of return words to $x_{[0,n)}$ and $V_0 (x) = {\mathcal A}$.
We stress the fact that $V_n(x)$ has cardinality $d$  for all $n$, according to Theorem~\ref{theo:return}.  
Let $(n_i)_{i \geq 1}$ be a strictly increasing integer sequence such that $n_1 = 1$ and such that each $v_{j,n_i}x_{[0,n_i)}$ occurs in $x_{[0,n_{i+1})}$ and in each $v_{k,n_{i+1}}$.
Let $\theta_i $ be an  endomorphism of ${\mathcal A}^*$ such that  $\theta_i ({\mathcal A}) = V_{n_i} (x)$.
Since $x_{[0,n_i)}$ is a prefix of $x_{[0,n_{i+1})}$, any $v_{j,n_{i+1}}\in V_{n_{i+1}}(x)$ has a unique decomposition as a concatenation of elements $v_{k,n_i}\in V_{n_i}(x)$.
More precisely, for any $v_{j,n_{i+1}}\in V_{n_{i+1}}(x)$, there is a unique sequence $(v_{k_j(1),n_i},v_{k_j(2),n_i},\dots,v_{k_j(\ell_j),n_i})$ of elements of $V_{n_i}(x)$ such that $v_{k_j(1),n_i} \cdots v_{k_j(\ell_j),n_i} = v_{j,n_{i+1}}$ and for all $m \in \{1,\dots,\ell_j\}$, $v_{k_j(1),n_i} \cdots v_{k_j(m),n_i} x_{[0,n_i)}$ is a prefix of $v_{j,n_{i+1}} x_{[0,n_{i+1})}$. 
This induces a unique endomorphism $\lambda_i$ of ${\mathcal A}^*$ defined by $\theta_{i+1} = \theta_i \circ \lambda_i$.
From the choice of the sequence $(n_i)_{i\geq 1}$, the matrices $M_{\lambda_i}$ have positive coefficients, so the sequence of morphisms $(\lambda_i)_{i \geq 1}$ is primitive.
Furthermore, as $x_{[0,n_{i+1}]}$ is prefix of each $v_{j,n_{i+1}}x_{[0,n_{i+1)}}$, there exists some $v \in V_{n_i}(x)$ such that $v_{k_j(1)}=v$ for all $j$.
In other words, the morphisms $\lambda_i$ are left proper.
Finally, from Theorem~\ref{theo:return}, the matrices $M_{\lambda_i}$ are unimodular.
Hence $(X,S)$ is $\S$-adic generated by the primitive directive sequence of unimodular left proper endomorphisms $\boldsymbol{\lambda} = (\lambda_i)_{i\geq 1}$.
We deduce from Lemma~\ref{lemma:proper} that minimal dendric subshifts are primitive unimodular proper $\S$-adic subshifts.
\end{proof}

Observe that using Lemma~\ref{lemma:aperiodicity} we recover that minimal dendric subshifts on at least two letters are aperiodic.

\section{Dimension groups of primitive unimodular proper $\S$-adic subshifts} \label{sec:proofs2}

In this section we first  prove a key result of this paper, namely Theorem \ref{theo:cohoword},  which   states that  $H(X,T)= C (X, \mathbb{Z})/ \beta C (X, \mathbb{Z})$ is generated as an additive group  by the classes  of the   characteristic functions of letter cylinders. 
We then  deduce a simple expression for  the dimension group of primitive unimodular proper $\S$-adic subshifts.

\subsection{From letters  to  factors}
We recall  that $\chi_U$ stands for the characteristic function of the set $U$.

\begin{theorem}
\label{theo:cohoword}
Let $(X,S)$ be a primitive unimodular proper $\S$-adic subshift.
Any function  $f\in C (X, \mathbb{Z})$ is cohomologuous to some integer linear combination of  the form  $\sum_{a\in \A} \alpha_a \chi_{[a]}\in C (X,\mathbb{Z})$. Moreover, the classes  $[\chi_{[a]}]$, $a\in \A$, are $\mathbb{Q}$-independent.
\end{theorem}

\begin{proof}
Let $\boldsymbol{\tau} = (\tau_n:\A^* \to \A^*)_{n\geq 1}$ be a primitive unimodular proper  directive sequence of $(X,S)$, hence $X = X_{\boldsymbol{\tau}}$.
Using Proposition~\ref{prop:minimalSSIprimitif}, all subshifts $(X_{\boldsymbol{\tau}}^{(n)},S)$ are minimal and aperiodic and $\min_{a \in \A} |\tau_{[1,n)}(a)|$ goes to infinity when $n$ increases.

Let us show that the group $H(X,S )=C(X, {\mathbb Z})/ \beta C(X,{\mathbb Z}) $ is spanned by the set   of  classes  of characteristic functions of  letter
cylinders
$\{ [\chi_{[a]}] \mid a \in \A \}$.
From Theorem~\ref{theo:BSTY} and using the fact that $(X,S)$ is minimal and aperiodic, one has,   for all  positive  integer   $n$, that  
$$
\mathcal{P}_n = \{ S^k \tau_{[1,n]} ([a]) \mid 0\leq k < |\tau_{[1,n]} (a)| , a \in \A \}
$$
is a finite  partition of $X$ into clopen sets.  This provides a   family of nested Kakutani-Rohlin tower partitions.
The morphisms of the directive sequence $\boldsymbol{\tau}$ being proper, for all $n$, there are letters $a_n, b_n$ such that all images $\tau_n (c)$, $c \in \A$, start with $a_n$ and end with $b_n$.
From this, it is classical to check that $(\mathcal{P}_n)_n$ generates the topology of $X$ (the proof is the same as~\cite[Proposition 14]{DHS:99} that  is concerned with  the particular case $\tau_{n+1}=\tau_n$ for all $n$).

We first claim that  $H(X,S ) $ is spanned by the set of classes     $\cup_n \Omega_n $, where
$$\Omega_n =  \{ [\chi_{\tau_{[1,n]} ([a])}] \mid a \in \A \} \quad n \geq 1.$$
In other words,  $H(X,S)$ is spanned  by the set of  classes  of  characteristic functions of  bases   of the  sequence of   partitions  $({\mathcal P})_n$.
It suffices to check that, for all $u^- u^+\in \La (X)$, the class $[\chi_{[u^-.u^+]}]$ is a linear integer combination of elements belonging
 to some $\Omega_n$.

Let  us check this assertion. Let $u^-u^+$ belong to $\La (X)$. 
Since $\min_{a \in \A} |\tau_{[1,n)}(a)|$ goes to infinity, there exists $n$ such that  $|u^-|,|u^+| < \min_{a\in \A} |\tau_{[1,n)} (a) |$.
The directive sequence $ \boldsymbol{\tau}$ being proper, there exist words $w,w'$ with  respective  lengths $|w|=|u^-|$ and $|w'|=|u^+|$
 such that all images $\tau_{[1,n]} (a)$ start with $w$ and end with $w'$.

Let $x \in [u^-. u^+]$. 
Let $a \in \A$ and $k \in \mathbb{N}$, $0\leq k < |\tau_{[1,n]} (a)|$, such that $x$ belongs to the atom $S^k \tau_{[1,n]} ([a])$.
Observing that $\tau_{[1,n]} ([a])$ is included in $[w'.\tau_{[1,n]} (a)w]$, this implies that the full  atom $S^k \tau_{[1,n]} ([a])$ is included in $[u^-. u^+]$.
Consequently $[u^-. u^+]$ is a finite union of atoms in $\mathcal{P}_n$.
But each characteristic function  of   an atom  of the form   $ S^k \tau_{[1,n]} ([a]) $   is   cohomologous to    $\chi_{\tau_{[1,n]} ([a])}$. The proof  works   as in the proof of Proposition \ref{prop:cob}.  This thus    proves the claim. 

Now we claim that each element of $\Omega_n $ is a linear integer combination of elements in $\{ [\chi_{[a]}] \mid a \in \A \}$.
More precisely,  let us  show that $\chi_{\tau_{[1,n]}([b])}$ is cohomologous to 
$$
\sum_{a\in \A}  (M_{\tau_{[1,n]}}^{-1})_{b,a}\chi_{[a]} .
$$

Let $a\in \A$ and $n\geq 1$.
One has $[a] = \cup_{B \in \mathcal{P}_n } (B \cap [a]) $ and thus $\chi_{[a]} $ is cohomologous to the map 
$$
\sum_{b\in \A}  (M_{\tau_{[1,n]}})_{a,b}\chi_{\tau_{[1,n]}([b])} ,
$$
by using  the fact  that the  maps $  \chi_{S^k \tau_{[1,n]} ([a]) }$   are   cohomologous to    $\chi_{\tau_{[1,n]} ([a])}$.
This means that for $U = ([\chi_{[a]} ] ) _{a\in \A} \in H (X,S)^\A$ and $V = ([\chi_{\tau_{[1,n]}([a])}])_{a\in \A} \in H (X,S)^\A$, one has 
$$
U = M_{\tau_{[1,n]}}V 
$$
and as a consequence $V = M_{\tau_{[1,n]}}^{-1}U$.
This proves the claim and the first part of the theorem.
 
To show the independence, suppose that  there exists some row vector $\alpha = (\alpha_a )_{a\in \mathcal{A}} \in \mathbb{Z}^\mathcal{A}$ such that $\sum_a \alpha_a   [\chi_{[a]}] = 0$.
Hence there is some $f\in C (X, \mathbb{Z} )$ such that $\sum_a \alpha_a   \chi_{[a]} = f\circ S - f$.   
We now  fix some $n$ for which  $f$ is constant on each atom of $\mathcal{P}_n$.
Observe that for all $x\in X$ and all $k \in \N$, one has $f (S^k x ) - f(x) = \sum_{j=0}^{k-1} \alpha_{x_j}$. 
Let $c\in \mathcal{A}$ and $x\in \tau_{[1,n+1]} ([c])$.
Then, $x$ and $S^{ |\tau_{[1,n+1]} (c)|} (x)$ belong to  $\tau_{[1,n]} ([a_{n+1}])$.
Hence, $f (S^{ |\tau_{[1,n+1]} (c)|} x ) - f(x) = 0 $, and thus
$$
(\alpha M_{\tau_{[1,n]}})_c  = \sum_{j=0}^{|\tau_{[1,n+1]} (c)|-1} \alpha_{x_j} = 0 .
$$ 
This holds for all $c$, hence $\alpha M_{\tau_{[1,n]}} = 0$, which yields $\alpha = 0$, by invertibility of the matrix $ M_{\tau_{[1,n]}}$.
\end{proof}

Observe that in the previous result, we can relax the assumption of minimality. Indeed, 
one checks that the same proof works if we assume that $(X,S)$ is  aperiodic  (recognizability then holds by \cite{BSTY})
 and that  $\min_{a \in \A} |\tau_{[1,n)}(a)|$ goes to infinity.

We now derive  two corollaries  from Theorem~\ref{theo:cohoword}   dealing respectively with   invariant measures  and with the image  subgroup.
Note that Corollary \ref{coro:measures}   extends a statement initially proved for interval exchanges \cite{FerZam:08}.
See also \cite{BHL:2019} for a similar result in the framework of  automorphisms of the free group and \cite{BHL:2020} for subshifts with finite rank.
\begin{corollary}\label{coro:measures}
Let $(X,S)$ be a primitive unimodular proper $\S$-adic subshift over the alphabet $\A$ and let $\mu, \mu' \in \mathcal M (X,S)$.
If $\mu$ and $\mu'$ coincide on  the  letters, then they are equal, 
that is, if $\mu([a]) = \mu'([a])$ for all a in $\A$, then $\mu(U) = \mu'(U)$, for any clopen subset $U$ of $X$.
\end{corollary}

\begin{corollary}\label{coro:freq}
Let $(X,S)$ be a primitive  unimodular proper $\S$-adic subshift over  the alphabet ${\mathcal A}$.
The image subgroup of $(X,S)$  satisfies

\begin{align*}
I(X,S) & = \bigcap_{\mu\in\mathcal{M}(X,S)}  \langle {\mu([a]) : a\in A} \rangle \\
& =   \left\lbrace \alpha : \exists (\alpha_a)_{a\in \mathcal{A}} \in \mathbb{Z}^\mathcal{A} , \alpha 
 = \sum_{a\in\mathcal{A}} \alpha_a \mu([a]) \ \forall \mu \in \mathcal{M} (X,S) \right\rbrace .
\end{align*}
\end{corollary}

The proof of Corollary \ref{coro:freq} uses the two descriptions of the image subgroup given in Proposition \ref{prop:GW} and Proposition \ref{prop:cob}.

In both corollaries,   the assumption of being proper can be dropped. The   proof  then  uses the measure-theoretical Bratteli-Vershik representation of the primitive unimodular $\S$-adic subshift given in \cite[Theorem 6.5]{BSTY}.

\subsection{An explicit description of the dimension group}
Theorem~\ref{theo:cohoword} allows  a precise description of the dimension group of primitive unimodular proper $\S$-adic subshifts.
Note that in the case of interval exchanges, one  recovers     the results   obtained in   \cite{Putnam:89}; see also \cite{Putnam:92,GjerdeJo:02}.

We  first  need  the following Gottschalk-Hedlund type   statement  \cite{GotHed:55}. 

\begin{lemma}[\cite{Host:95}, Lemma 2, \cite{DHP}, Theorem 4.2.3]
\label{lemma:positivecohomologous}
Let $(X,T)$ be a minimal Cantor system and let $f \in C(X,\Z)$.
There exists $g \in C(X,\N)$ that is cohomologous to $f$ if and only if for every $x \in X$, the sequence $\left( \sum_{k=0}^n f \circ T^k(x)\right)_{n \geq 0}$ is bounded from below.
\end{lemma}
We now  can deduce   one of  our  main statements.   
\begin{theorem}\label{theo:dg}
Let $(X,S)$ be a primitive unimodular proper $\S$-adic subshift over a $d$-letter alphabet. 
The linear map $\Phi : H (X,S) \to \mathbb{Z}^d$ defined by $\Phi (  [\chi_{[a]}] ) = e_a$, where $\{ e_a \mid a \in \mathcal{A} \}$ is the canonical base of $ \mathbb{Z}^d$,
defines an isomorphism of dimension groups from $K^{0}(X,S)$ onto
\begin{align}
\label{align:ZdGD}
\left( {\mathbb Z} ^d, \,  \{ {\bf x} \in {\mathbb Z}^d \mid \langle {\bf x},  \boldsymbol{ \mu} \rangle > 0 \mbox{ for all }  
 \mu \in  {\mathcal M} (X,S) \}\cup \{{\mathbf 0}\},\,  \bf{1}\right ),
 \end{align}
 where the entries of $\bf{1}$ are equal to $1$. 
\end{theorem}  
\begin{proof}
From Theorem \ref{theo:cohoword}, $\Phi$ is well defined and is a group isomorphism from $H (X,S)$ onto $\Z^d$.
We obviously have $\Phi (  [1] ) = \Phi(\sum_{a \in \A} [\chi_{[a]}]) = \bf{1}$ and it remains to show that 
\[
\Phi(H^+(X,S)) 
= 
\{ {\bf x} \in {\mathbb Z}^d \mid \langle {\bf x},  \boldsymbol{ \mu} \rangle > 0 \mbox{ for all }  
 \mu \in  {\mathcal M} (X,S) \}
 \cup 
 \{{\mathbf 0}\}.
\]
Any element of $H^+(X,S)$ is of the form $[f]$ for some $f \in C(X,\N)$.
From Theorem \ref{theo:cohoword}, there exists a unique vector ${\bf x} = (x_a)_{a \in \A}$ such that $[f] = \sum_{a \in \A} x_a [\chi_{[a]}]$.
As $f$ is non-negative, we have, for any $\mu \in  {\mathcal M} (X,S)$, 
\[
	\langle	\Phi([f]),\boldsymbol{\mu} \rangle
	=
	\sum_{a \in \A} x_a \mu([a])
	=
	\int f d\mu
	\geq 0,
\]
with equality if and only if $f=0$ (in which case ${\bf x} = {\bf 0}$).

For the other inclusion, assume that ${\bf x} = (x_a)_{a \in \A} \in {\mathbb Z}^d$ satisfies $\langle {\bf x},  \boldsymbol{ \mu} \rangle > 0 \mbox{ for all }  
 \mu \in  {\mathcal M} (X,S)$ (the case ${\bf x} = {\bf 0}$ is trivial).
We consider the function $f = \sum_{a \in \A} x_a \chi_{[a]}$. According to   Lemma~\ref{lemma:positivecohomologous},  the existence of  $f' \in [f]$ such that $f'$ is non-negative is  equivalent to  the   existence of   a   lower bound for  ergodic sums.
Assume by contradiction that there exists a  point $x \in X$ such that the sequence $\left( \sum_{k=0}^n f \circ S^k(x)\right)_{n \geq 0}$ is not bounded from below.
Thus there is a an increasing sequence of positive integers $(n_i)_{i \geq 0}$ such that 
\[
	\lim_{i \to +\infty} \sum_{k=0}^{n_i-1} f \circ S^k(x) = -\infty.
\]
Passing to a subsequence $(m_i)_{i \geq 0}$ of $(n_i)_{i \geq 0}$ if necessary, there exists $\mu \in \mathcal{M}(X,S)$ satisfying
\[
	\langle {\bf x}, \boldsymbol{\mu} \rangle = \int f d\mu = \lim_{i \to +\infty} \frac{1}{m_i} \sum_{k=0}^{m_i-1} f \circ S^k(x) \leq 0,
\]
which contradicts our hypothesis.
The sequence $\left( \sum_{k=0}^n f \circ S^k(x)\right)_{n \geq 0}$ is thus bounded from below and we conclude  by using Lemma~\ref{lemma:positivecohomologous}.
\end{proof}

\begin{remark}
We cannot remove the hypothesis of being left or right proper in Theorem \ref{theo:dg}.
Consider indeed  the subshift $(X,S)$ defined by the  primitive unimodular non-proper substitution $\tau$  defined over $\{a,b\}^*$ as  $\tau \colon   a   \mapsto aab,  b \mapsto ba$.
According to   \cite[p.114] {DurandThese},  the dimension group  of $(X,S)$  is  isomorphic to 
$\left( {\mathbb Z}^3 , \left\{ {\mathbf x} \in {\mathbb Z}^3  :    \langle {\mathbf x},  {\mathbf v}  \rangle  > 0 \right\} , (2,0,-1) \right)$
where ${\mathbf v}=  ( \frac{1+ \sqrt 5}{2}, 2,1)$.

\end{remark}

\subsection{Ergodic measures}\label{subsection:ergodic}
We now focus on further   consequences of Theorem \ref{theo:cohoword}  for  invariant measures of  primitive unimodular proper $\S$-adic subshifts.

\begin{corollary}\label{cor:oe}
Two primitive unimodular proper $\S$-adic subshifts $(X_1,S)$ and $(X_2,S)$ are strong orbit equivalent 
if and only if
there is a unimodular matrix $M$ such that $M {\bf 1} = {\bf 1}$ and 
\[
	\{\boldsymbol{\nu} \mid \nu \in \mathcal M(X_2,S)\}
	=
	\{M^{\rm T} \boldsymbol{\mu} \mid \mu \in \mathcal M(X_1,S)\}.	
\]
In particular, $(X_1,S)$ and $(X_2,S)$ are defined on alphabets with the same cardinality.
\end{corollary}
\begin{proof}
For $i = 1,2$, let $\Phi_i:H(X_i,S) \to \Z^{d_i}$ be the map given in Theorem~\ref{theo:dg}, where $d_i$ is the cardinality of the alphabets $\A_i$ of $X_i$.
Let us also write ${\bf 1}_i$ the vector of dimension $d_i$ only consisting in $1$'s and
\[
	C_i 
	= 
	\Phi_i(H^+(X_i,S))
	=
	\{ {\bf x} \in {\mathbb Z}^{d_i} 
	\mid 
	\langle {\bf x},  \boldsymbol{ \mu} \rangle > 0 
	\mbox{ for all } \mu \in  {\mathcal M} (X_i,S) \}
	\cup \{{\mathbf 0}\},	
\]
so that $\Phi_i$ defines an isomorphism of dimension groups from $K^0(X_i,S)$ onto $(\Z^{d_i},C_i,{\bf{1}}_i)$.

First assume that $(X_1,S)$ and $(X_2,S)$ are strong orbit equivalent.
Theorem~\ref{oe} implies that there is an isomorphism of dimension group from $(\Z^{d_2},C_2,{\bf{1}}_2)$ onto $(\Z^{d_1},C_1,{\bf{1}}_1)$.
Hence $d_1 = d_2 = d$ and this isomorphism is given by a unimodular matrix $M$ of dimension $d$ satisfying $M{\bf 1} = {\bf 1}$ (where ${\bf 1} = {\bf 1}_1 = {\bf 1}_2$) and $MC_2 = C_1$. 
We also denote by $M$ the map ${\bf x} \in \Z^d \mapsto M {\bf x}$.

Recall from Section~\ref{subsec:cyl} that the map
\[
\mu \in \mathcal M(X_i,S) \mapsto \left(\tau_\mu: [f] \in H(X_i,S) \mapsto \int f d\mu\right)
\]
is an affine isomorphism from $\mathcal M(X_i,S)$ to $\T(K^0(X_i,S))$.
Observing that for all $\mu \in \mathcal M(X_1,S)$, $\tau_\mu \circ \Phi_1^{-1} \circ M \circ \Phi_2$ is a trace of $K^0(X_2,S)$, it defines an affine isomorphism
$\mu \in \mathcal M(X_1,S) \mapsto \nu \in \mathcal M(X_2,S)$, where $\nu$ is such that $\tau_\nu = \tau_\mu \circ \Phi_1^{-1} \circ M \circ \Phi_2$.
Since $\boldsymbol{\mu} = (\tau_\mu([\chi_{[a]}]))_{a \in \A_1}$, we have, for all $a \in \A_2$,
\[
	\nu([a]) 
	= \tau_\nu( [\chi_{[a]}]) 
	= \tau_\mu \circ \Phi_1^{-1} \circ M \circ \Phi_2 ( [\chi_{[a]}])
	= \boldsymbol{\mu}^{\rm T} M e_a
	= e_a^{\rm T} M^{\rm T} \boldsymbol{\mu}, 
\] 
so that $\boldsymbol{\nu} = M^{\rm T} \boldsymbol{\mu}$.

Now assume that we are given a unimodular matrix $M$ satisfying $M {\bf 1} = {\bf 1}$ and 
\[
	\{\boldsymbol{\nu} \mid \nu \in \mathcal M(X_2,S)\}
	=
	\{M^{\rm T} \boldsymbol{\mu} \mid \mu \in \mathcal M(X_1,S)\}.	
\]
In particular, this implies that $d_1 = d_2=d$.
Let us show that the map $M: {\bf x} \in \Z^d \mapsto M {\bf x}$ defines an isomorphism of dimension groups from $(\Z^{d},C_1,{\bf{1}})$ to $(\Z^{d},C_2,{\bf{1}})$.
We only need to show that $MC_1 = C_2$.
The matrix $M$ being unimodular, we have $M {\bf x} = {\bf 0}$ if and only if ${\bf x} = {\bf 0}$.
For ${\bf x} \neq {\bf 0}$, we have 
\begin{align*}
	{\bf x} \in C_2 
	& \Leftrightarrow \langle {\bf x},\boldsymbol{\nu} \rangle >0 \text{ for all } \nu \in \mathcal{M}(X_2,S)	\\
	& \Leftrightarrow \langle {\bf x},M^{\rm T}\boldsymbol{\mu} \rangle >0 \text{ for all } \mu \in \mathcal{M}(X_1,S)	\\
	& \Leftrightarrow \langle M {\bf x},\boldsymbol{\mu} \rangle >0 \text{ for all } \mu \in \mathcal{M}(X_1,S)	\\
	& \Leftrightarrow M {\bf x} \in C_1,
\end{align*}
which ends the proof.
\end{proof}


According to  Theorem \ref{theo:dg},  dimension groups of primitive unimodular proper subshifts have  rank $d$.
This implies that the number $e$  of  ergodic measures  satisfies  $e\leq d$. 
In fact,  we have even more from the following result. 

\begin{proposition}\label{prop:effrosshen81}
\cite[Proposition 2.4]{Effros&Shen:1981}
Finitely generated simple dimension groups of rank $d$ have at most $d-1$ pure traces.
\end{proposition}
Dimension groups of minimal Cantor systems $(X,T)$ are simple dimension groups (Theorem \ref{theo:simpleminimal}) and,   since the Choquet simplex of traces is affinely isomorphic to the simplex of ergodic measures,  we derive the following.
\begin{corollary}\label{cor:d-1}
Primitive unimodular proper $\S$-adic subshifts over a $d$-letter alphabet  have at most $d-1$ ergodic measures.
\end{corollary}
If the primitive unimodular proper $\S$-adic subshift $(X,S)$ has some extra combinatorial properties, then  the number of ergodic measures can   be smaller. 
Suppose indeed  that $(X,S)$ is a minimal dendric subshift on a $d$-letter alphabet.
As its factor complexity equals   $(d-1)n + 1$,  one has  a priori $e \leq  d-2$ for $d \geq 3$  according to  \cite[Theorem 7.3.4]{CANT}.
One can even   have more  as a direct consequence of  \cite{Damron:2019} and  \cite{DolPer:2019}. Note that  this statement  encompasses    the case of interval exchanges handled in  \cite{Katok:73,Veech:78}. 

\begin{theorem} \label{theo:ergodic}
Let $(X,S)$ be a minimal dendric subshift over a $d$-letter alphabet. 
One has 
$$\mbox{\rm Card}({\mathcal M}_e (X,S)) \leq  \frac{d}{2}.$$
\end{theorem} 
\begin{proof}
According  to \cite{Damron:2019}, a  minimal subshift is  said to  satisfy the 
 regular bispecial condition if any large enough bispecial word $w$ has only one left extension $aw \in \La(X)$, $a \in \A$, that is right special and only one right extension $wa \in \La(X)$, $a \in \A$, that is left special. Now we use the fact  that   minimal  dendric subshifts   satisfy  the regular bispecial condition according to  \cite{DolPer:2019}.
 We conclude  by using the  upper  bound on  the number ergodic measures from  \cite{Damron:2019}.
\end{proof} 

\section{Infinitesimals and balance property} \label{sec:saturation}


When the infinitesimal  subgroup  $ \mbox{\rm Inf} (K^0 (X,T))$ of a minimal Cantor system $(X,T)$ is trivial, the system is called {\em saturated}. This  property is proved  in \cite{TetRob:2016} to hold for  primitive, aperiodic, irreducible substitutions  for  which images of letters  have a common prefix.
At the opposite, an example of a dendric  subshift with non-trivial  infinitesimal subgroup is provided in Example \ref{ex:nontrivial}. 
 A  formulation of saturation in terms of the  topological full group is given in   \cite{BezKwia:00}. Recall also that  for saturated systems, the quotient group $I(X,T)/E(X,T)$ is torsion-free by \cite[Theorem 1]{CortDP:16}  (see also \cite{GHH:18}).

We  first state   a  characterization of  the triviality of the infinitesimal subgroup $\mbox{\rm Inf} (K^0 (X,S))$ for minimal   unimodular   proper $\S$-adic  subshifts   (see  Proposition \ref{theo:saturated}).  We then relate    the saturation property  with  a combinatorial notion called {\em balance property}   and  we provide   a  topological  characterization of primitive
unimodular proper $\S$-adic subshifts that are balanced  (see Corollary \ref{cor:balanced}).

\begin{proposition}\label{theo:saturated}
Let $(X,S)$ be a minimal   unimodular   proper $\S$-adic  subshift on a $d$-letter alphabet $A$.
The   infinitesimal  subgroup  $\mbox{Inf} (K^0 (X,S))$  is   non-trivial  
if  and only if there is a   non-zero vector ${\mathbf x} \in {\mathbb Z}^d$
orthogonal to any element of the simplex of letter measures. 

\noindent In particular, if  there exists  some invariant measure $\mu \in {\mathcal M}(X,S)$  for which the frequencies of letters $\mu([a])$, $a \in A$, are rationally independent, then the   infinitesimal  subgroup  $\mbox{Inf} (K^0 (X,S))$  is   trivial.
 \end{proposition}
 \begin{proof}
 According to Theorem \ref{theo:dg},  the elements of  $ \mbox{\rm Inf} (K^0 (X,S))$ are the  classes of functions  that are represented by  vectors ${\mathbf x} \in {\mathbb Z}^d$ such that 
 $\langle  {\mathbf x}, \boldsymbol{ \mu} \rangle  =0$ for every $\mu \in {\mathcal M} (X,S) $. 
Recall also  that coboundaries    are represented by the vector ${\mathbf 0}$. 
Hence   $\mbox{Inf} (K^0 (X,S))$  is  not trivial   if  and only if there exists  $ {\mathbf x} \in {\mathbb Z}^d$, with  ${\mathbf x} \neq {\mathbf 0}$,
such that $\langle {\bf x},  \boldsymbol{ \mu} \rangle =  0$,  for every   $\mu \in \M(X,S)$.

 Assume now that  there exists  some invariant measure $\mu \in {\mathcal M}(X,S)$  for which the frequencies of letters    are rationally independent.
Hence,  for any vector  ${\mathbf x} \in {\mathbb Z}^d$, $\langle {\bf x},  \boldsymbol{ \mu} \rangle =  0$  implies  that   ${\mathbf x}={\mathbf 0}$.
From above, this implies that  $\mbox{Inf} (K^0 (X,S))$  is   trivial.
\end{proof}

See  Example \ref{ex:nontrivial}   for an example of a dendric subshift with non-trivial infinitesimals.

We now introduce  a notion of  balance for functions. Let $(X,T)$ be a minimal Cantor system. 
We say  that $f \in C (X , \mathbb{R} )$ is {\em balanced} for $(X,T)$    whenever  there exists a  constant $C_f>0$ 
such that $$|\sum_{i=0}^n f(T^ix) - f(T^iy) | \leq C_f \mbox { for all  }x, y \in X \mbox{ and for all } n .$$

Balance property  is usually expressed for   letters and factors  (see  for instance 
\cite{BerCecchi:2018}). Indeed a    minimal  subshift   $(X,S)$ is said to be {\em balanced   on the factor} $v \in {\mathcal L} (X)$  if $\chi_{[v]}: X \to  \{0,1\}$ is balanced,  or, equivalently, 
if there exists a constant $C_v$ such that for  all  $w,w'$  in $\mathcal{L}_X$ with $|w|=|w'|$,  then     $||w|_v-|w'|_v|\leq C_v.$  
It is   {\em  balanced on  letters} if it is balanced on each letter,  and   it is  {\em balanced on   factors}   if it is balanced on all  its   factors.

More generally, we say that a system $(X,T)$ is balanced on a subset $H \subset C(X, \mathbb{R})$ whenever it is balanced for all $f$ in $H$. 
It is standard to check that any system $(X,T)$ is balanced on the (real)  coboundaries.
Of course, a subshift $(X,S)$ is balanced on a generating set of $C(X,\mathbb{Z})$ if and only if  it is balanced on factors   or, equivalently, 
  if every  $f\in C(X,\mathbb{Z})$ is balanced.
  
One can observe that the balance property on letters is not necessarily preserved under topological conjugacy whereas the balance property  on factors is.
Indeed, consider the shift  generated by the  Thue--Morse substitution $\sigma \colon a \mapsto ab, b \mapsto ba$. 
It is clearly balanced on letters. 
It is  conjugate to the shift generated by the substitution
$\tau \colon  a \mapsto bb$, $b\mapsto bd$, $c \mapsto  ca$, $d\mapsto cb$  via the  sliding block code
$00  \mapsto a,$  $01 \mapsto b,$ $10  \mapsto c,$  $11  \mapsto d$ (see \cite[p.149]{Queffelec:2010}).
The   subshift generated by $\tau$ is  not balanced  on letters  (see \cite{BerCecchi:2018}). 

Next   proposition will be useful to characterize balanced functions of a system $(X,T)$.
\begin{proposition}\label{prop:carBalanced}
Let $(X,T)$ be a minimal dynamical system. An integer valued continuous function $f  \in C(X, {\mathbb Z})$ is balanced for $(X,T)$ if and only if there exists  $\alpha \in {\mathbb R}$
such that  the map $f-\alpha$ is a real coboundary.  
In this case, $\alpha = \int f d\mu$, for any $T$-invariant probability measure $\mu$ in $X$. 
\end{proposition}

\begin{proof}
If the function $f-\alpha$ is a real coboundary, one easily checks  that $f$ is balanced. 
Moreover, the   integral with respect to any $T$-invariant probability measure is zero, providing the last claim.

Assume that $f  \in C(X, {\mathbb R})$ is balanced for $(X,T)$. Let $C>0$  be a constant such that
 $  |\sum_{i=0}^n f\circ T^{i}(x)-f\circ T^{i}(y)| \le C$ holds uniformly in $x,y \in X$  for all  $n\ge 0$.  
Thus,  for any non-negative integer $p \in \N$, there exists $N_{p}$ such that, for any $x\in X$, one has the following inequalities:
$$N_{p} \le \sum_{i=0}^p f\circ T^i(x) \le N_{p} +C.$$
Moreover, one checks that,  for any $p,q \in \N$:
$$
qN_{p} \le \sum_{i=0}^{pq} f\circ T^i(x) \le qN_{p} +qC \textrm{ and }   pN_{q} \le \sum_{i=0}^{pq} f\circ T^i(x) \le pN_{q} +pC.
$$
It follows that $-qC \le qN_{p} -pN_{q} \le pC$ and thus $-C/p \le N_{p}/p -N_{q}/q \le C/q $.  Hence the sequence $(N_{p}/p)_{p}$ is a Cauchy sequence. 
Let $\alpha = \lim_{p\to \infty} N_{p}/p$. By letting $q$ going to infinity, we get $ -C \le N_{p}-p\alpha \le 0$, so that  $-C \le \sum_{i=0}^p f\circ T^i(x) -p\alpha \le C$ for any $x\in X$. By the classical   Gottschalk--Hedlund's Theorem \cite{GotHed:55}, the function $f-\alpha$ is a  real coboundary.
\end{proof}

As a corollary, we deduce that  a minimal Cantor  system $(X,T)$ balanced on $C(X,{\mathbb Z})$  is uniquely ergodic. 
It also follows that  for a   minimal  subshift $(X,S)$   balanced on the factor $v$,    the  frequency $\mu_{v} \in {\mathbb R}_{+}$ of $v$ exists,
i.e., 
for any    $x \in X$,   
$\lim_{ n  \rightarrow \infty}  \frac{ | x_{-n}  \cdots  x_0 \cdots x_{n}| _v }{  2n+1}=  \mu_v$,  and even,   the quantity $\sup _{n \in \mathbb{N}} | | x_{-n}  \cdots  x_0 \cdots x_{n}| _v -  (2n+1)  \mu_v| $ is finite (see  also \cite{BerTij:2002}). 

Actually, integer-valued continuous functions that are balanced for a minimal Cantor system $(X,T)$ are  related to the continuous eigenvalues of the system as illustrated by the following folklore  lemma.  We recall that  $E(X,T)$ stands for the set of additive  continuous eigenvalues.

\begin{lemma}\label{lem:BalancedEigenvalue}
Let $(X,T)$ be a minimal Cantor system  and let $\mu$ be a $T$-invariant measure. If $f \in C(X,\Z)$ is balanced  for $(X,T)$, then $\int f d\mu $ belongs to $ E(X,T)$. 
\end{lemma}
\begin{proof}
If $f \in C(X,\Z)$ is  balanced  for $(X,T)$, then so is $-f$ and there exists $g \in  C(X,\mathbb{R})$  such that 
$ -f+ \int f d\mu = g \circ T-g$ (by Proposition \ref{prop:carBalanced}).
This yields $\exp({2i\pi g\circ T})= \exp ({2i\pi  \int f d \mu }) \exp ({ 2i \pi  g})$ by noticing that   $\exp({- 2i\pi  f (x)})=1$ for  any $x\in X$. 
Hence $\exp ({ 2i \pi  g})$ is a continuous   eigenfunction  associated with the  additive  eigenvalue $\int f d \mu$.
 \end{proof}

We first  give a statement valid  for any  minimal  Cantor system  that will then  be applied below to primitive unimodular proper $\S$-adic subshifts.
We recall from \cite[Theorem 3.2, Corollary 3.6]{GHH:18} that there exists a one-to-one homomorphism $\Theta $ from $I (X,T)$ to $K^0 (X,T)$ such that, for $\alpha \in (0,1)\cap E (X,T)$, $\Theta (\alpha ) = [\chi_{U_{\alpha}}]$   where $U_\alpha$ is a clopen set,  sucht that $\mu (U_\alpha)=\alpha$ for  every
invariant measure $\mu$,   and  $ \chi_{U_{\alpha}} -\mu (U_{\alpha}) $ is a  real coboundary. Hence  $\chi_{U_{\alpha}}$ is balanced for $(X,T)$ (by Proposition \ref{prop:carBalanced}). 
 \begin{proposition}
\label{prop:balanced}
Let $(X,T)$ be a  minimal  Cantor system. 
The following are equivalent:
\begin{enumerate}
\item
\label{item:1}
$(X,T)$ is balanced on some  $H\subset C(X,\mathbb{Z})$ and $\{ [h] : h \in H \}$ generates $K^0 (X,T)$,
\item
\label{item:2}
$(X,T)$ is balanced on $C( X, \mathbb{Z})$,
\item
\label{item:3}
$\Theta (E(X,T))$ generates $K^0 (X,T)$.
\end{enumerate}
In this case  $(X,T)$ is uniquely ergodic, then $I(X,T) = E(X,T)$ and ${\rm Inf } (K^0(X,T))$ is trivial. 
\end{proposition}

\begin{proof}
Let us prove that   \eqref{item:1} implies  \eqref{item:2}.
Let $f\in C (X,\mathbb{Z} )$.
One has $[f] = \sum_{i=1}^{n} z_i [h_i]$ for some integers $z_i$ and some functions $h_i\in H$.
Hence $f = g\circ T - g + \sum_{i=1}^{n} z_i h_i$ for some $g\in C(X,\mathbb{Z} )$ and $f$ is balanced. 
Consequently, $(X,T)$ is balanced on $C(X,\mathbb{Z})$.
It is immediate that \eqref{item:3}  implies \eqref{item:1}.

Let us  show  that \eqref{item:2} implies \eqref{item:3}.   Unique ergodicity holds by Proposition~\ref{prop:carBalanced}.
 Let $\mu$ be the unique  shift invariant  probability  measure of $(X,T)$.
For any $f\in C( X,\mathbb{Z})$, there are clopen sets $U_i$ and integers $z_{i}$ such that $f = \sum_{i=1}^n z_i \chi_{U_i}$. 
From Lemma \ref{lem:BalancedEigenvalue}  the values $ \mu(U_{i})\in I(X,T)$ are additive continuous eigenvalues in $E(X,T)$. 
We get $[\chi_{U_{i}}] = \Theta(\mu(U_{i}))$  and  $[f] = \sum_{i=1}^n z_i \Theta (\mu(U_{i}) )$.
This shows the claim \eqref{item:3}.

Assume that one of  the three  equivalent conditions holds.  
Let $\mu$ denote the unique shift invariant probability measure.
Then, any map $f - \int f d \mu $,  with $f\in C (X,\mathbb{Z} )$, is a real coboundary (by Proposition \ref{prop:carBalanced}). Hence, since any integer valued continuous  function that is a real coboundary is a  coboundary (\cite[Proposition 4.1]{Ormes:00}), the infinitesimal subgroup $ \mbox{\rm Inf} (K^0 (X,T))$ is trivial.
Moreover, Lemma \ref{lem:BalancedEigenvalue} implies that $I(X,T) \subset E(X,T)$.  
The reverse implication   $E(X,T) \subset I(X,T)$ comes from  \cite[Proposition 11]{CortDP:16}, see also   \cite[Corollary 3.7]{GHH:18}.
\end{proof}
 
Observe that when $(X,S)$ is a minimal subshift, by taking $H$ 
to be the set of classes of characteristic functions of cylinder sets, the balance
property  is equivalent to the algebraic condition \eqref{item:3} of  Proposition  \ref{prop:balanced}.

We  now  provide a topological  proof  of the fact that  the balance property on letters implies the balance property on  factors
for  primitive unimodular proper $\S$-adic subshifts.
For  minimal  dendric subshifts, this was already proved  in \cite[Theorem 1.1]{BerCecchi:2018}  using a combinatorial proof.

\begin{corollary} \label{cor:balanced}
Let $(X,S)$ be a  primitive  unimodular proper $\S$-adic subshift on a $d$-letter alphabet. 
The following are equivalent:
\begin{enumerate}
\item
\label{item:balancedmaps}
$(X,S)$ is balanced for all  integer valued   continuous maps in  $C (X , \mathbb{Z} )$,
\item
\label{item:balancedfactors}
$(X,S)$ is balanced on factors,
\item
\label{item:balancedletters}
$(X,S)$ is balanced on  letters, 
\item
\label{item:balancedrank}
$\mbox{\rm rank}  (E(X,S)) =d$,
\end{enumerate}
and in this case $(X,S)$ is uniquely ergodic, $I(X,S) = E(X,S)$ and $\mbox{\rm Inf} (X,S)$ is trivial. 
\end{corollary}

\begin{proof}
The  implications  $\eqref{item:balancedmaps} \Rightarrow \eqref{item:balancedfactors}\Rightarrow \eqref{item:balancedletters}$  are immediate. 
Let us prove  the  implication $\eqref{item:balancedletters} \Rightarrow \eqref{item:balancedrank}$.  
We deduce from Theorem \ref{theo:cohoword}, by taking $H$ 
to be the set of classes of characteristic functions of cylinder sets,  that the conditions of  Proposition \ref{prop:balanced} hold.
We deduce from Proposition \ref{theo:saturated} that  \eqref{item:balancedrank} holds.

It remains to prove   the  implication $\eqref{item:balancedrank} \Rightarrow \eqref{item:balancedmaps}$. Suppose that  $E(X,S)$ has rank  $d$.
 Let $\alpha_1 , \dots , \alpha_d\in E(X,S)$ be rationally independent. 
 There is no restriction to assume   that they are all in $(0,1)$. 
Consider, for $i =1\cdots,d$,  $\Theta (\alpha_i ) = [\chi_{U_{\alpha_i}}]$   where $U_{\alpha_{i}}$ is a clopen set such that  $\mu (U_{\alpha_{i}} ) = \alpha_i$ for any $S$-invariant measure $\mu \in {\mathcal M}(X,S)$ and $\chi_{U_{\alpha_{i}}}$ is balanced for $(X,S)$.
The  classes  $\Theta (\alpha_i)$'s are rationally independent because
the image   of $\Theta (\alpha_i)$ by any trace  is $\alpha_{i}$ and  these values are assumed to be rationally independent. 
%
As $K^0 (X,S)$ has rank $d$, by Theorem \ref{theo:cohoword},  and  since it has no torsion (as recalled in Section~\ref{subsec:dim}), any element $[f]\in K^0 (X,S)$ is a rational linear combination of the $\Theta (\alpha_i )$'s. By Proposition \ref{prop:balanced}, any $f\in C(X, \Z )$ is balanced for $(X,S)$.
 \end{proof}
 
 \begin{remark}We deduce that primitive unimodular proper $\S$-adic subshifts  that are  balanced on letters have  the maximal  continuous eigenvalue  group property,  as defined in \cite{Durand&Frank&Maass:2019}, i.e.,  $E(X,S)= I(X,S)$.  This implies  in particular   that  non-trivial additive eigenvalues are irrational.
 Indeed,   by Corollary \ref{cor:balanced},  $\mbox{Inf} (K^0 (X,S))$ is trivial and  thus, by  Proposition \ref{theo:saturated},  the   frequencies of letters  (for the unique
 shift-invariant measure) are rationally independent,  which yields  that  $I(X,S)$ and  thus $E(X,S)$ contain no rational non-trivial elements. 
 
More generally, the fact that  non-trivial additive eigenvalues are irrational hold  for minimal dendric subshifts (even without 
the  balance property)  \cite{Rigidity}.  Note also  that the triviality of  $\mbox{Inf} (K^0 (X,S))$  says nothing  about  the balance property (see Example \ref{ex:balance}),
but the existence of non-trivial infinitesimals indicates that some letter is not balanced.
Lastly, the Thue--Morse substitution  $\sigma \colon a \mapsto ab, b \mapsto ba$  generates 
 a   subshift that is  balanced on letters but not on factors \cite{BerCecchi:2018}.  This substitution is neither  unimodular,  nor proper.
 \end{remark}

\section{Examples  and observations }\label{sec:examples}


\subsection{Brun subshifts}\label{subsec:Brun}
We  provide a family of  primitive unimodular proper $\S$-adic subshifts which are  not dendric.
We consider the set of endomorphisms $S_\mathrm{Br} = \{\beta_{ab} \mid a \in \mathcal{A},\, b \in \mathcal{A} \setminus \{a\}\}$ over $d$~letters defined by 
\[
\beta_{ab}:\ b \mapsto ab,\ c\mapsto c\ \text{for}\ c \in \mathcal{A} \setminus \{ b\}.
\]

A subshift  $(X,S)$ is a \emph{Brun subshift} if it is generated by a primitive directive sequence $\boldsymbol{\tau } = (\tau_n)_n \in S_\mathrm{Br}^\mathbb{N}$ 
such that for all $n$ the endomorphism $\tau_n\tau_{n+1}$ belongs to
\begin{align*}
\big\{\beta_{ab}\beta_{ab} \mid a \in \mathcal{A},\, b \in \mathcal{A} \setminus \{a\}\big\}  
\cup 
\big\{\beta_{ab}\beta_{bc} \mid a \in \mathcal{A},\, b \in \mathcal{A} \setminus \{a\},\, c \in \mathcal{A} \setminus \{b\}\big\}.
\end{align*}
Observe that primitiveness of $\boldsymbol{\tau }$ is equivalent to the fact that for each $a \in \mathcal{A}$ there is $b \in \mathcal{A}$ such that $\beta_{ab}$ occurs infinitely often in $\boldsymbol{\tau }$. 
Brun subshifts are not dendric in general: on a three-letter alphabet, they may contain strong and weak bispecial factors, hence that have an extension graph which is not a tree~\cite{Labbe&Leroy:16}. 
However, we show below that they are primitive unimodular proper $\S$-adic subshifts.

\begin{lemma}
\label{lemma:brunproper}
Let $\mathcal{A}$ be a finite alphabet and $\gamma_{ab} : \mathcal{A} \to \mathcal{A}$, $a\not = b$, be  the  letter-to-letter map  defined by $\gamma_{ab} (a) = \gamma_{ab} (b) = a$ and $\gamma_{ab} (c)  =c$ for $c\in \mathcal{A}\setminus \{ a,b\}$.
Let $(a_n)_{1\leq n\leq N}$  be such that $\{ a_n \mid 1\leq n \leq N \} = \mathcal{A}$.
Then, $\gamma_{a_1 a_{2}} \gamma_{a_2 a_{3}} \cdots  \gamma_{a_{N-1} a_{N}}$ is constant.
\end{lemma}
\begin{proof}
It suffices to observe that $\gamma_{a_m a_{m+1}} \gamma_{a_{m+1} a_{m+2}} \cdots  \gamma_{a_{n-1} a_{n}}$ identifies the letters $a_m, a_{m+1}, a_{m+2}  \dots , a_n$ to $a_m$.
\end{proof}

\begin{lemma} 
Brun subshifts are primitive unimodular proper $\S$-adic subshifts.
\end{lemma}
\begin{proof}
Let $(X,S)$ be a Brun subshift over the alphabet $\mathcal{A}$, generated by the directive sequence $\boldsymbol{\beta } = (\beta_{a_nb_n})_{n\geq 1} \in S_\mathrm{Br}^\mathbb{N}$.
With each endomorphism $\beta_{ab}$ one can associate the map $\gamma : \mathcal{A} \to \mathcal{A}$ defined by $\gamma (c) = \beta (c)_0$.
Clearly $\gamma $ is equal to $\gamma_{ab}$.

By primitiveness, there exists an increasing sequence of integers $(n_k)_k$, with $n_0 =0$, such that $\{ a_i \mid n_k \leq i < n_{k+1} \} = \mathcal{A}$ for all $k$.
Hence from Lemma \ref{lemma:brunproper} the morphisms $\boldsymbol{\beta}_{[n_k , n_{k+1})}$ are left proper. 
We conclude by  using Lemma~\ref{lemma:proper}.
\end{proof}

As a corollary, we recover the following result (which also follows from~\cite[Theorem~5.7]{Berthe&Delecroix:14}).

\begin{proposition}\label{prop:UE}
Brun subshifts are uniquely ergodic.
\end{proposition}
\begin{proof}
This follows from Corollary~\ref{coro:measures} and from the fact that Brun subshifts have a simplex of letter measures generated by a single vector (see~\cite[Theorem~3.5]{Brentjes:81}).
\end{proof}

Brun  subshifts have been introduced in \cite{BST:2019} in order to provide  symbolic models for  two-dimensional toral  translations.
In particular,  they   are proved to  have  generically pure discrete spectrum in  \cite{BST:2019}.

\subsection{Arnoux-Rauzy  subshifts} \label{subsec:AR}

A minimal subshift $(X,S)$ over $ \mathcal{A}= \{1,2,\ldots,d\}$ is an \emph{Arnoux-Rauzy  subshift} if  for all~$n$ it has $(d-1)n+1$ factors of length~$n$, with exactly one left special and one right special factor of length~$n$. 
Consider the following set of endomorphisms defined on the alphabet $\mathcal{A} = \{1, \dots , d \}$, namely 
$S_\mathrm{AR} = \{ \alpha_a \mid a \in \mathcal{A}\}$ with
\[
	\alpha_a:\ a \mapsto a,\ b \mapsto ab\ \mbox{for}\ b \in \mathcal{A} \setminus \{a\}.
\]

A subshift $(X,S)$ generated by a primitive directive sequence $\boldsymbol{\tau } \in S_\mathrm{AR}^\mathbb{N}$ is called an  \emph{Arnoux-Rauzy subshift}.
It is standard to check  that primitiveness of $\boldsymbol{\alpha }$ is equivalent to the fact that each morphism $\alpha_a$ occurs infinitely often in $\boldsymbol{\alpha }$. 
Arnoux-Rauzy subshifts being dendric subshifts, they are in particular primitive unimodular proper $\S$-adic subshifts.
We similarly recover, as in  Proposition~\ref{prop:UE}, that   Arnoux-Rauzy  subshifts are uniquely ergodic (see~\cite[Lemma 2]{Delecroix&Hejda&Steiner:2013} for
the fact that  Arnoux-Rauzy subshifts  have a simplex of letter measures generated by a single vector).
\subsection{A dendric subshift with non-trivial infinitesimals} \label{ex:nontrivial} 
Let us provide an example of a  minimal  dendric subshift with  non-trivial  infinitesimal subgroup, and thus with  rationally  dependent letter  measures according to
Proposition \ref{theo:saturated}.
We take the interval exchange $T$ with permutation $(1,3,2)$ with intervals $[0,1-2 \alpha) , [1-2 \alpha, 1- \alpha), $  and $[1-\alpha,1)$, with $\alpha=(3-\sqrt{5})/2$.
The transformation $T$ is represented in Figure~\ref{figure3interval}, with 
$I_{1}=[0,1-2\alpha), $ $I_{2}=[1-2\alpha,1-\alpha)$, $ I_{3}=[1-\alpha,1)$ 
and $
J_{1}=[0,\alpha)$, $J_{2}=[\alpha,2\alpha)$, $J_{3}=[2\alpha,1).$

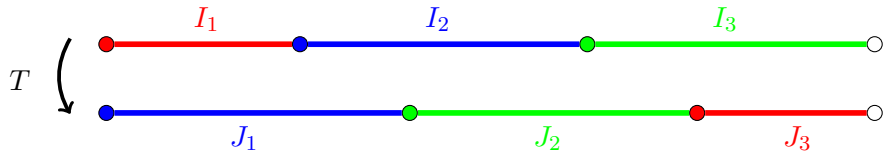
\begin{figure}[hbt]
 \tikzset{node/.style={circle,draw,minimum size=0.2cm,inner sep=0pt}}
 \tikzset{title/.style={minimum size=0.5cm,inner sep=0pt}}
 \begin{center}
  \begin{tikzpicture}
   \node[node,fill=red](0h) {};
   \node[node,fill=blue](1h) [right = 2.36cm of 0h] {};
   \node[node,fill=green](2h) [right = 6.18cm of 0h] {};
   \node[node](3h) [right = 10cm of 0h] {};
   \node[node,fill=blue](0b) [below= 0.7cm of 0h] {};
   \node[node,fill=green](1b) [right = 3.82cm of 0b] {};
   \node[node,fill=red](2b) [right = 7.64cm of 0b] {};
   \node[node](3b) [right = 10cm of 0b] {};
   \draw[line width=0.75mm, red] (0h) edge node {} (1h);
   \draw[line width=0.75mm, blue] (1h) edge node {} (2h);
   \draw[line width=0.75mm, green] (2h) edge node {} (3h);
   \draw[line width=0.75mm, blue] (0b) edge node {} (1b);
   \draw[line width=0.75mm, green] (1b) edge node {} (2b);
   \draw[line width=0.75mm, red] (2b) edge node {} (3b);
   \color{black}
   \node[title](aa) [above right = 0cm and 1cm of 0h,red] {$I_1$};
   \node[title](ba) [above right = 0cm and 1.5cm of 1h,blue] {$I_2$};
   \node[title](ca) [above right = 0cm and 1.5cm of 2h,green] {$I_3$};
   \node[title](cb) [below right = 0cm and 1.5cm of 0b,blue] {$J_1$};
   \node[title](bb) [below right = 0cm and 1.5cm of 1b,green] {$J_2$};
   \node[title](ab) [below right = 0cm and 1cm of 2b,red] {$J_3$};
   
   \node[title](a) [above left = 0cm and 0cm of 0h] {};
   \node[title](b) [below = 1cm of a] {};
   \node[title](T) [below left = 0.3cm and 0.3 of a] {$T$};
   \draw[line width=0.5mm, bend right, ->] (a) edge node {} (b);
  \end{tikzpicture}
 \end{center}

 \caption{The transformation $T$.}
 \label{figure3interval}
\end{figure}

Measures  of letters are rationally dependent and  the  natural coding of this   interval exchange is  a strictly ergodic  dendric  subshift $(X,S,\mu)$ by Theorem \ref{theo:ergodic}.
It is actually a representation on $3$ intervals of the rotation of angle $2\alpha$
(the point $1-\alpha$ is a separation point which is not a singularity of  this interval exchange).
 One has
$\mu([2])=\mu([3])$.
The  class of the function   $\chi_{[2]}- \chi_{[3]}$  is  thus a   non-trivial infinitesimal, according to Theorem \ref{theo:cohoword}.

\subsection{Dendric subshifts  having  the same dimension group and different spectral properties  } \label{ex:balance}

It is well known that within any given class of  strong orbit equivalence (i.e.,  by Theorem \ref{oe},   within any  family of minimal Cantor systems sharing the same dimension group $(G,G^+,u)$), all minimal Cantor systems share the same set  of rational   additive   continuous eigenvalues $E(X,T)\cap \mathbb{Q}$ \cite{Ormes:97}. 
When this set is reduced to $\{ 0 \}$, then, in the strong orbit equivalence class of $(X,T)$, there are many weakly mixing systems, see \cite[Theorem 6.1]{Ormes:97}, \cite[Theorem 5.4]{GHH:18} or \cite[Corollary 23]{Durand&Frank&Maass:2019}.

We provide here  an  example of a  strong orbit equivalence class  that  contains two minimal dendric subshifts, one being  weakly mixing and the other one  having pure discrete spectrum.  
Both systems  are  saturated (they have no non-trivial  infinitesimals)  but  they have different balance properties.
They  are defined on  a  three-letter  alphabet and  have factor complexity $2n + 1$.
According to Theorem \ref{theo:ergodic},  they  are uniquely ergodic. 
From Corollary~\ref{cor:oe}, two  minimal  dendric  subshifts on a three-letter alphabet  are strong orbit equivalent if and only if there is a unimodular row-stochastic matrix $M$ sending the vector of letter measures of one subshift to the vector of letter measures of the other.
In particular, any Arnoux-Rauzy subshift is strong orbit equivalent to any natural coding of an i.d.o.c. exchange of three intervals for which the length of the intervals are given by the letter measures of the Arnoux-Rauzy subshift (recall that an interval exchange transformation satisfies the {\it infinite distinct orbit condition}, i.d.o.c. for short, if the negative trajectories of the discontinuity points are infinite disjoint sets; this condition  implies minimality \cite{Keane:75}). We  thus consider the subshift $(X,S)$ generated by the Tribonacci substitution
$\sigma \colon a \mapsto ab, b \mapsto ac, c \mapsto a$ which is uniquely ergodic, dendric, balanced  and has discrete spectrum \cite{Rauzy:1982}.
Let $\mu$ be its unique invariant measure. 
We also consider the natural coding $(Y,S)$ of the three-letter  interval exchange defined on intervals of length $\mu[a]$,  $\mu[b]$,  $\mu[c]$ with permutation $(13)(2)$.
It is uniquely ergodic, topologically weakly mixing  \cite{KatokStepin,FerHol:2004} and strong orbit equivalent to $(X,S)$ by Proposition \ref{prop:DendricareSadic}. 
Hence,  for spectral reasons,  $(X,S)$ and  $(Y,S)$  
 are not  topologically conjugate, even if they are strong orbit equivalent.

We    provide  a further proof of   non-conjugacy   for  the systems $(X,S)$ and $(Y,S)$   based on asymptotic pairs.
We first recall    a few  definitions.
Two  points $x,y $ in a given subshift  are  said to be  {\em right asymptotic}  if  they have  a common tail, i.e., 
 there exists $n $ such that  $(x_k)_{k \geq n}= (y_k)_{k \geq n}$.
 This defines an equivalence  relation   on  the collection of orbits:
 two $S$-orbits  ${\mathcal O}_S(x) = \{ S^n x \mid n\in \mathbb{Z} \} $ and   ${\mathcal O}_S(y)$   are  asymptotically equivalent 
 if  for  any $x' \in {\mathcal O}_S(x)$, there is $y' \in  {\mathcal O}_S(y)$ that is  right asymptotic to $x'$.
We call  {\em asymptotic component} any equivalence  class  under the asymptotic  equivalence. 
We say that  it is {\em non-trivial} whenever it is not reduced to one orbit.

An Arnoux-Rauzy  subshift $(X,S)$ has a unique non-trivial asymptotic component formed  of three distincts orbits as, for all $n$, there is a unique word $w$ of length $n$ such that $\ell(w) \geq 2$ and this word is such that $\ell(w)=3$ (see Section \ref{subsec:tree} for the notation).
On the other side, any i.d.o.c. exchange of  three-intervals $(Y,S)$  has  2 asymptotic components and thus cannot be   conjugate to $(X,S)$.
Indeed, suppose  that it has a unique non-trivial asymptotic component. 
As a natural coding of an i.d.o.c interval exchange transformation has two left special factors for each large enough length, 
this component should contain three sequences $x'x$, $ x''ux$ and $x''' ux$ belonging to $Y$ where $u$ is a non-empty word. 
This would imply that  the interval exchange transformation is not i.d.o.c.

Next statement illustrates  the  variety of spectral behaviours  within  strong orbit equivalence classes of dendric   subshifts. 
 \begin{proposition}\label{prop:realization}
For Lebesgue a.e. probability vector   $\boldsymbol{ \mu}$   in ${\mathbb R}^3_+$,
there  exist   two strictly ergodic  proper unimodular $\S$-adic  subshifts, one     with  pure    discrete spectrum and  another one   which  is   weakly mixing, both having the same dimension group 
  $\left( {\mathbb Z} ^3, \,  \{ {\bf x} \in {\mathbb Z}^3 \mid \langle {\bf x},  \boldsymbol{ \mu} \rangle > 0 \}\cup \{{\mathbf 0}\},\,  \bf{1}\right ).$
  \end{proposition}
  \begin{proof}
Brun   subshifts  such as  introduced in Section \ref{subsec:Brun} are proved   to have   generically   pure discrete spectrum
in 
\cite{BST:2019}. 
See \cite{KatokStepin,FerHol:2004}  for  the  genericity of   weak mixing for subshifts generated by   three-letter  interval exchanges. 
\end{proof}
\subsection{Dendric  vs.   primitive unimodular proper $\mathcal{S}$-adic subshifts }\label{subsection:vs}
In this   section,  we give an example of a primitive unimodular proper $\mathcal{S}$-adic subshift  whose strong orbit equivalence class   contains no minimal  dendric subshift.
Theorem~\ref{theo:dg} provides a description of the dimension group of any primitive unimodular proper $\mathcal{S}$-adic subshift.
It is  natural to ask whether a strong orbit equivalence class represented by such a dimension group includes a primitive unimodular proper $\mathcal{S}$-adic subshift. 
This was conjectured in different terms in \cite{Effros&Shen:1979}.
It was shown to be true when the dimension group has a unique trace~\cite{Riedel:1981b} (or, equivalently,  when all minimal systems in this class are uniquely ergodic)  but shown to be false in general~\cite{Riedel:1981}. In the same spirit, one may ask if the strong orbit equivalence class of any primitive unimodular proper $\mathcal{S}$-adic subshift contains a dendric subshift.
Inspired from \cite{Effros&Shen:1979}, we negatively answer to that question below.

Indeed, this  example provides  a family of examples of  primitive unimodular $\S$-adic subshifts on a three-letter alphabet  with two ergodic invariant probability  measures.
They thus    cannot be dendric by  Theorem~\ref{theo:ergodic} and   their strong orbit equivalence class contains no minimal dendric subshift.

Let ${\mathcal A}=\{1,2,3\}$ and consider the directive sequence $\boldsymbol{\tau } = (\tau_n : {\mathcal A}^* \to {\mathcal A}^*)_{n \geq 1}$ defined by 
\begin{align*}
\tau_{2n}:
&
\
1 \mapsto 2^{a_n}3, \quad 
2 \mapsto 1, \quad
3 \mapsto 2 \\
\tau_{2n+1}:
& 
\ 
1 \mapsto 32^{a_n}, \quad 
2 \mapsto 1, \quad
3 \mapsto 2 
\end{align*}
where $(a_n)_{n \geq 1}$ is an increasing sequence of positive integers satisfying $\sum_{n \geq 1} 1/a_n  <1$.
The incidence matrix of each morphism $\tau_n$ is the unimodular matrix
 $$
A_n =
\begin{pmatrix}
0 & 1 & 0 \\
a_n & 0 & 1 \\
1 & 0  & 0 \\
\end{pmatrix}.
$$
It is easily checked that any morphism $\tau_{[n,n+5)}$ is proper and has an incidence matrix with positive entries.
Therefore, the $\S$-adic subshift $(X_{\boldsymbol{\tau}} , S)$ is primitive, unimodular and proper.
Let us show that $(X_{\boldsymbol{\tau}} , S)$ has two ergodic measures.    In fact, we  prove that $(X_{\boldsymbol{\tau}} , S)$  has at least two ergodic measures.
This will  imply  that  it has  exactly two ergodic measures by Theorem \ref{prop:effrosshen81}. 

For all $n \geq 1$, let $C_n$ be the first column vector of $A_{[1,n]} = A_1 \cdots A_n$ and $J_n = C_n/\Vert C_n\Vert_1$ where $\Vert\cdot\Vert_1$ stands for the L1-norm.
If $(X_{\boldsymbol{\tau}} , S)$ had only one ergodic measure $\mu$, then $(J_n)_{n \geq 1}$ would converge to $\boldsymbol{\mu}$. 
Hence it suffices to show that $(J_n)_{n \geq 1}$ does not converge.

Observe that, using the shape of the matrices $A_n$, that for $n\geq 1$, and setting $C_0 = e_1, C_{-1} = e_2, C_{-2} = e_3$,  one has
$$
C_n = a_n C_{n-2} + C_{n-3}, 
$$ 
where $e_1,e_2,e_3$ are the canonical vectors.
Hence we have $J_n = b_n J_{n-2} + c_nJ_{n-3}$ with 
$b_n = a_n \frac{\Vert C_{n-2}\Vert_1}{\Vert C_{n}\Vert_1}$ and 
$c_n = \frac{\Vert C_{n-3}\Vert_1}{\Vert C_{n}\Vert_1}$.
In particular, $b_n + c_n = 1$.
As $(\Vert C_n\Vert_1)_n$ is non-decreasing, we have 
$1 \geq b_n \geq a_n c_n$ 
and thus $c_n \leq a_n^{-1}$.
Moreover, we have
\[
\Vert J_n - J_{n-2} \Vert_1 = \Vert(b_n - 1) J_{n-2} -c_n J_{n-3} \Vert_1 = c_n \Vert J_{n-2} -  J_{n-3} \Vert_1 \leq \frac{2}{a_n},
\]
hence, for $0\leq m\leq n$,
$$
\Vert J_{2n}-J_{2m}\Vert_1 \leq 2 \sum_{k=m+1}^{n} \frac{1}{a_{2k}}.
$$
This shows that $(J_{2n})_{n \geq 1}$ is a Cauchy sequence. Let  $\beta$  stand for  its limit.
For $n \geq m=0$, we obtain 
$\Vert J_{2n}-e_1\Vert_1 \leq 2 \sum_{k=1}^{n} \frac{1}{a_{2k}}$ and thus 
$\Vert\beta-e_1\Vert_1 \leq 2 \sum_{k=1}^{\infty } \frac{1}{a_{2k}}$.

We similarly show that $(J_{2n+1})_{n \geq 1}$ is a Cauchy sequence. Let $\alpha $  stand for  its limit. We have $\Vert\alpha-e_2\Vert_1 \leq 2 \sum_{k=0}^{\infty } \frac{1}{a_{2k+1}}$.
Consequently, 
$$
\Vert\alpha - \beta \Vert_1 = 
\Vert(\alpha -e_2) + (e_1- \beta ) +e_2 -e_1 \Vert_1 \geq 2 - 2 \sum_{k=1}^{\infty } \frac{1}{a_{k}} 
$$
and $(J_n)_{n \geq 1}$ does not converge. 
Consequently, $(X_{\boldsymbol{\tau}} , S)$ has exactly two ergodic measures.

\section{Questions and further works}

 According to \cite{Riedel:1981} (see also Section \ref{subsection:vs}), not all strong orbit equivalence classes represented by dimension groups of the type \eqref{align:ZdGD} in Theorem~\ref{theo:dg} contain primitive unimodular proper $\mathcal{S}$-adic subshifts. 
The description of the dynamical dimension group in Theorem~\ref{theo:dg}  is not precise enough to explain the restrictions that occur for instance for the measures,
so that a complete characterization of the  dynamical dimension groups of primitive unimodular proper $\mathcal{S}$-adic subshifts is still missing.
 
Similarly, we address the question of characterizing the strong orbit equivalence classes containing minimal dendric subshifts. 
The combinatorial properties of these subshifts imply constraints, especially for the  invariant measures, such as stated in Theorem \ref{theo:ergodic}.  
For example, the question  arises as to whether   dimension groups of rank $d$ having at most $d/2$ extremal traces  are dimension groups of minimal dendric subshifts.

Another question is about the properness  assumption.
For dendric or Brun subshifts, we were able to find a primitive unimodular proper $\mathcal{S}$-adic representation. 
One can easily define $\mathcal{S}$-adic subshifts by a primitive unimodular directive sequence that is not proper. 
The question  now    is whether 
 a primitive unimodular proper $\mathcal{S}$-adic representation  (up to conjugacy) of this subshift can be found.
Even in the substitutive case, we do not know whether such a representation exists.

The factor complexity of dendric subshifts is  affine. 
It is well known~\cite{Boyle&Handelman:94,Ormes:97,Sugisaki:2003} that inside the strong orbit equivalence class of any minimal Cantor system one can find another minimal Cantor systems with any other prescribed topological entropy. 
Primitive unimodular proper $\mathcal{S}$-adic subshifts being of finite topological rank,  they have zero topological entropy. 
It would be interesting to  exhibit  a variety of asymptotic behaviours   for   complexity functions within a  strong orbit equivalence class.

\bibliographystyle{alpha}
\bibliography{DimG}

\end{document}